\documentclass[12pt, twoside, letterpaper]{amsart}
\usepackage{amsmath, hyperref, color, amssymb}
\usepackage{srcltx}
\usepackage{ulem}

\usepackage{geometry}
\theoremstyle{plain}%
\newtheorem{Theorem}{Theorem}[section] %
\newtheorem{Corollary}[Theorem]{Corollary}%
\newtheorem{Lemma}[Theorem]{Lemma}

\newtheorem{innercustomthm}{Theorem}[section]
\newenvironment{customthm}[1]
  {\renewcommand\theinnercustomthm{#1}\innercustomthm}
  {\endinnercustomthm}
  
  \newtheorem{innercustomlm}{Lemma}[section]
\newenvironment{customlemma}[1]
  {\renewcommand\theinnercustomlm{#1}\innercustomlm}
  {\endinnercustomlm}
  
\theoremstyle{definition}%
\newtheorem{Assumption}[Theorem]{Assumption}%
\newtheorem*{Acknowledgments}{Acknowledgments} %
\theoremstyle{remark}%
\newtheorem{Remark}[Theorem]{Remark} %

\newcommand{\eps}{\varepsilon}

\newcommand{\Pas}{\text{$\mathbb{P}$--a.s.}}
\newcommand{\esssup}{\operatorname*{\mathrm{ess\,sup}}}
\newcommand{\argmax}{\operatorname*{\mathrm{arg\,max}}}
\newcommand{\essinf}{\operatorname*{\mathrm{ess\,inf}}}

\newcommand{\tbl}[1]{\textcolor{black}{#1}}

\newcommand{\dbracc}[1]{[\kern-0.15em[ #1 ]\kern-0.15em]}
\newcommand{\dbraco}[1]{[\kern-0.15em[ #1 [\kern-0.15em[}
\newcommand{\dbraoc}[1]{]\kern-0.15em] #1 ]\kern-0.15em]}
\newcommand{\dbraoo}[1]{]\kern-0.15em] #1 [\kern-0.15em[}
\newcommand{\be}{\begin{equation}}
\newcommand{\ee}{\end{equation}}
\newcommand{\ba}{\begin{aligned}}
\newcommand{\ea}{\end{aligned}}

\DeclareMathOperator{\conv}{conv}

\usepackage{cancel}
\begin{document}
\title[Stability of the indirect utility process]{Stability of the indirect utility process}

\author{Oleksii Mostovyi}
\address{Oleksii Mostovyi, Department of Mathematics, University of Connecticut, Storrs, CT 06269, United States}%
\email{oleksii.mostovyi@uconn.edu}%
\date{\today}
\maketitle

\begin{abstract}
We investigate the dynamic stability of the indirect utility process associated with a (possibly suboptimal) trading strategy under perturbations of the market. Establishing the  \tbl{reverse conjugacy} characterizations first, we prove continuity and first-order convergence of the indirect-utility process under simultaneous perturbations of the finite variation and martingale parts of the return of the \tbl{risky asset}. 
\end{abstract}

\section{Introduction}

Indirect utility appears in mathematical finance as one of the primary criteria of the quality of a portfolio. Therefore, the regularity of the solution to certain classical problems in mathematical finance is associated with the stability or continuity of the indirect utility under perturbations of the initial data. From the practical viewpoint, as every statistical procedure allows for an only approximate determination of the model parameters, implementation of algorithms of optimal investment and utility-based pricing and hedging hinges on continuity of the indirect utility under model perturbations. The results of this paper show that under reasonably natural assumptions the indirect utility is a stable criterion of quality of the portfolio. Further, this also holds in dynamic settings, \tbl{where the dynamic characterizations are usually harder to establish, and even for suboptimal portfolios.}


Mathematically, the results below hinge on the dual characterization of the indirect utility associated with a possibly suboptimal trading strategy. Both the dynamic formulation and such a suboptimality lead to  multiple difficulties related to establishing convex-analytic results \tbl{for} functions whose codomain is a space of random variables $\mathbb L^0$ (and not $\mathbb R$), i.e., for random elements. It turns out that even fundamental theorems of convex analysis are harder to establish in such settings, and for example, the classical Fenchel-Moreau (or biconjugation) theorem over a {\it dual pair of Banach spaces} have been proven only recently, see \cite{Kupper17}. \tbl{In the present settings, however, we need not work with a dual pair, but rather with a pair of {\it polar sets} of stochastic processes, where polarity has to be understood appropriately. We identify precisely such polar sets and use the classical characterizations of the sets of wealth processes from \cite{Delbaen-Schachermayer1997} combined with changes of num\'eraire and results from \cite{Mostovyi2015} to obtain the biconjugation result.} 

\tbl{
In the process of proving the biconjugacy, we establish a result, which is closely related to the conditional minimax theorem. Another version of this theorem can be found in \cite{ErhanKravitz}. Note that in our formulation we also do not require compactness of either domain, but only {\it boundedness in probability} is needed.} Such boundedness often appears in the mathematical finance literature in both primal and dual domains under natural no-arbitrage conditions, for example in \cite{KS}. Note that minimax without compactness is a classical subject of  analysis, see e.g., \cite{minimaxFan}, \cite{HaMinimax}, \cite{minimax2}; a version of the minimax theorem that is helpful in financial applications, in particular below, and that does not require compactness of either domain can be found in \cite{bankKauppila}.

\tbl{For the stability analysis, we introduce a parametrization of perturbations, which allows considering distortions of the drift or volatility of the \tbl{risky asset} together or separately. 
 Then we identify certain primal and dual feasible elements under such perturbations  
and  prove convergence of the indirect utility process (in the sense of Theorem \ref{thmDeriv} below) and complement this with finding its  associated derivative with respect to a parameter (also in Theorem \ref{thmDeriv}).} Dual characterization is key here.

Our construction of the dual domain is consistent with the weak no-arbitrage condition, no unbounded profit with bounded risk (NUPBR) introduced in \cite{KarKar07}, that still allows for the meaningful structure of the underlying problem.  \textcolor{black}{As the formulation of the indirect utility in a  dynamic formulation is closely related to forward  performance processes (FPPs) of the form \cite[Definition 4.3]{ZZ}}, one of the contributions of the present paper is in showing that FPPs on a finite time horizon can be considered under NUPBR and possibly without stronger no free lunch with vanishing risk (NFLVR) condition, which was predominantly used for the investigation of FPPs in the past. Note that FPPs were originally introduced in \cite{MZ1} and \cite{MZ2} to measure performances of portfolios in a way that allows for dynamic adaptation of the investor's preferences. Thus, the results of this paper provide an approach for the analysis of FPPs under market perturbations. They also imply the robustness of the indirect utility (and therefore the FPP in finite-horizon settings) as a dynamic criterion of the quality of a portfolio, and thus this paper complements the research of many authors, in particular, \cite{RMT}, \textcolor{black}{in non-Markovian settings}. However, the complete analysis of general FPPs goes beyond the scope of the current paper.
\textcolor{black}{Note that, in static settings, questions related to stability were investigated in \cite{AZ} and \cite{KardarasZitkovic2011}, among others.}

The remainder of the paper is organized as follows: in section \ref{secModel} we specify the model, section \ref{secDualChar} contains the dual characterization, in section \ref{secStab} we show the stability of the indirect utility process.
\section{Model}\label{secModel}
We work on a filtered probability space $(\Omega, \mathcal F, (\mathcal F_t)_{t\geq 0}, \mathbb P)$, where the filtration satisfies the usual conditions, $\mathcal F_0$ is trivial. There is a \tbl{riskless asset}, whose price equals to $1$ at all times, and a \tbl{risky one}. The conditions on the \tbl{risky asset} will alter in different sections, thus in section \ref{secDualChar}, it can be considered to be a general multidimensional semimartingale, and for stability analysis in section \ref{secStab} and below, we will work with a $1$-dimensional continuous process. 
\subsection{Utility field and maximization problem}
Let us consider an Inada stochastic utility field $U$: $[0,\infty)\times\Omega\times [0,\infty)\to \mathbb R\bigcup\{-\infty\}$, i.e., a stochastic field, which satisfies the following assumption. 
\begin{Assumption}\label{asInada}
For every $(t,\omega)\in [0,\infty)\times \Omega$, $U(t,\omega,\cdot)$ is an Inada utility function, that is, a strictly increasing, strictly concave, differentiable function, which satisfies the Inada conditions:
$$\lim\limits_{x\searrow 0}U'(t,\omega,x) = \infty\quad {and} \quad \lim\limits_{x\to\infty}U'(t,\omega,x) = 0,$$
where $U'$ denotes the derivative with respect to the last argument. 
At $x = 0$, we suppose that $U(t,\omega, 0) = \lim\limits_{x\searrow 0} U(t,\omega, x)$. This value may be $-\infty$. We also suppose that $U(\cdot,\cdot,x)$ is optional for every $x>0$. As it is common in the probability literature, the symbol $\omega$ will usually be omitted. 
\end{Assumption}
We refer to \cite[Chapter 9]{Jarrow} for an overview of utility functions.  
In this and the following section, we consider only one market, where there is a $d$-multidimensional \tbl{risky asset} with a return process $R^0$ and a riskless asset, whose price equals to $1$ at all times. 
Following \cite{KS}, we denote by $\mathcal X(x)$ the set of nonnegative wealth processes:
$$\mathcal X(x) = \left\{X\geq 0:~X = x + H\cdot R^0~for~some~R^0-integrable~H\right\},\quad x\geq 0.$$
In this market, we  fix an initial wealth $\bar x\geq 0$ and a predictable and $R^0$-integrable process $\pi$ \tbl{(up to $t\in[0,T]$)}, which specifies the proportions of wealth invested in corresponding \tbl{\tbl{risky asset}s} and such that $X^{\pi} = \bar x\mathcal E\left( \pi\cdot R^0\right)\geq 0,$
where, \tbl{here and below,} $\mathcal E(\cdot)$ denotes the stochastic exponential.
The set of wealth processes in $\mathcal X(\bar x)$, which equal to  $X^\pi$ on $[0,t]$, is denoted by $\mathcal A(X^{\pi}_t, t)$, that is
\begin{equation}\label{defA}
\mathcal A(X^{\pi}_t, t) :=\left\{\tilde X\in\mathcal X(\bar x):~ \tilde X_s = X^\pi_s,~for~s\in[0,t],~\mathbb P-a.s.
\right\}.
\end{equation}
In such settings, we \textcolor{black}{define  an indirect utility 
 up to $T$} of  $\pi$ as
\begin{equation}\label{defu}
u(X^{\pi}_t, t, T) :=\esssup\limits_{\tilde X\in\mathcal A(X^{\pi}_t, t)}\mathbb E\left[ U\left(T, \tilde X_T\right)|\mathcal F_t\right], \quad t\in[0,T].
\end{equation}
Note that this definition is closely related to the definition of forward performance processes, see, e.g., \cite{Tehran1}, \cite{SergeyThaleia1}, and
\cite{MishaRonnieLevon} on a finite time horizon. However, \eqref{defu} does not require the existence of the optimizer (instead, it is proven below), and the supermartingale structure of 
$u(X^{\pi}_t, t, T)$, $t\in[0,T]$, i.e., 
$\mathbb E\left[ u(X^{\pi}_{t_2}, t_2, T)|\mathcal F_{t_1}\right]\leq u(X^{\pi}_{t_1}, t_1, T)$ for $0\leq t_1\leq t_2\leq T$, can also be shown.  \tbl{Further,  \eqref{defu} is  also forward performances in the sense of \cite[Definition 4.3]{ZZ}, where the difference is in the exact form of the domain for the optimization problem \eqref{defu}.}
\subsection{Reformulation of \eqref{defu}}
For the analysis below, we need to extend the definition of \eqref{defu} to the closure of the convex solid hull of $\{X_t:~X\in\mathcal X(x)\}$. This is done in the following two-step procedure. First, we define
\begin{equation}\label{defCt}
\mathcal C_t(x) \triangleq\left\{g\in m\mathcal F_t:~g\leq X_t ~for~some~X\in\mathcal X(x)\right\},~~x\geq 0,
\end{equation} 
\tbl{where  $m\mathcal  F_t$ stands for $\mathcal F_t$ measurability of $g$, {see e.g., \cite[{p. 29}]{Williams} for   notations of this kind.} Now, we set}
\begin{equation}\label{defCtT}
\mathcal C_{tT}(x) :=\left\{g\in m\mathcal F_T:~g\leq x + \int_{t+}^TH_udR^0_u,~for~some~R^0-integrable~H\right\},~~ x\geq 0,
\end{equation}
and we will denote $\mathcal C_t(1)$ by $\mathcal C_t$ as well as $\mathcal C_{tT}(1)$ by $\mathcal C_{tT}$, respectively. 
\begin{Remark}\textcolor{black}{The idea behind such definitions is that every stochastic integral of the form $x + \int_0^THdR^0$ can be represented as $(x+\int_0^tH_udR^0_u)(1 + \int_{t+}^T\tilde H_udR^0_u)$ for an appropriate $\tilde H$. Also, similarly to the argument in \cite[Section 4, p. 926]{KS} 
one can show that the sets $\mathcal C_t$ and $\mathcal C_{tT}$ are closed with respect to the topology of convergence in measure. This, in particular, allows for the following representation of $\mathcal C_T$: $$\mathcal C_T = \mathcal C_t\mathcal C_{tT} = \left\{\xi\rho:~\xi\in\mathcal C_t,~\rho\in\mathcal C_{tT}\right\},\quad \tbl{t\in[0,T]}.$$
 Further, as $U$ is increasing, \tbl{an} optimizer to \eqref{defu} is a maximal element of $\mathcal C_{tT}$. \tbl{Thus,} by enlarging the domain of \eqref{defu} \tbl{as above}, we do not lose  the structure of the solution to \eqref{defu}. On the other hand, by passing from the set of wealth processes to the closure of the (convex) solid hull of such processes we gain the properties needed, in particular, for the conjugacy characterization below.} 
\end{Remark}
Thus, we extend the definition of $u$ in \eqref{defu} from $X^\pi$ to $\mathcal C_t$ as follows 
\begin{equation}\label{primalProblem}
u( \xi,t,T) :=\esssup\limits_{\rho\in\mathcal C_{tT}}\mathbb E_t\left[U(T,\xi\rho)\right],\quad \xi\in\bigcup\limits_{x\geq 0}\mathcal C_{t}(x),
\end{equation} 
with the version of an effective domain of $u(\cdot, t,T)$ being the set 
\begin{equation}\label{effDomP}
\tbl{\tbl{\mathcal E^u_t}} :=
\left\{\xi\in\bigcup\limits_{x\geq 0}\mathcal C_{t}(x): ~there~exists~\rho\in\mathcal C_{tT}, ~such~that~\mathbb E\left[U^-(T, \xi\rho)\right]<\infty\right\}.
\end{equation}

\subsection{Technical Assumptions}
We will impose the following conditions. 
\begin{Assumption}\label{NUPBR} (NUBPR up to $T$) Let $T>0$ be fixed. The set
$\left\{X_T:~X\in\mathcal X(1)\right\}$ is bounded in probability. 
\end{Assumption}
\begin{Assumption}\label{finiteness}(fin value at $T$)
Let $T>0$ be fixed. We suppose that
$$u(z,0,T)>-\infty\quad {\rm and} \quad \sup\limits_{x>0}\left( u(x,0,T) - xz\right)<\infty,\quad z>0 .$$
\end{Assumption}
These conditions are necessary of the model to be nondegenerate, see e.g., the abstract theorems in \cite{Mostovyi2015} and \cite{MostovyiNUPBR}. \tbl{In the notations of section \ref{secStab}, these conditions will be imposed on the base or, equivalently, $0$-model.}
\begin{Remark}
Assumption \ref{finiteness} will imply that the conditional expectations below are well-defined, where we adopt the definition \cite[Definition 1, p. 211]{ShirProb}, which does not require integrability. 
\end{Remark}

\section{Dual Characterization}\label{secDualChar}
In this section, we will suppose that the prices process for the \tbl{risky asset}  is a general multidimensional semimartingale, \tbl{not necessarily continuous}. \tbl{The main contributions of this section are in finding the right structure of the dual problem that, in dynamic settings, allows for the existence and uniqueness results and a biconjugacy characterization of the indirect utility under no unbounded profit with bounded risk.  Further, the results of this section provide a version of the Fenchel-Moreau theorem for random elements, i.e., functions whose codomains is a space of random variables. Note that this topic is not well-studied, and a version of a Fenchel-Moreau theorem over a pair of Banach spaces is only recently proven in \cite{Kupper17}. Additionally, by using a change of num\'eraire approach below, a minimax type of result is established without the compactness of either domain. The polar structure of the primal and dual domains is key here, though. }
\subsection{Existence and uniqueness of a solution to \eqref{defu}}
Let us denote 
\begin{equation}\label{defZt}
\begin{split}
{\mathcal Z}_t  := &\left\{(Z_s)_{s\in[t,T]}\geq 0:\right.\\
&\left.~Z_t\leq 1~and~(Z_sX_s)_{s\in[t,T]}~is~a~supermartinglale~for~every~X\in\mathcal X(1)\right\},\\
\end{split}
\end{equation}
and recall that the notion of \textcolor{black}{Fatou-convergence} of stochastic processes, is introduced in \cite[Definition 5.2]{FK}.
The following lemma shows existence of an optimizer to \eqref{primalProblem}.
\begin{Lemma}\label{231}
Let $T>0$ be fixed and let us suppose that Assumptions \ref{asInada}, \ref{NUPBR}, and \ref{finiteness} hold. 
Then for every $\xi\in\tbl{\mathcal E^u_t}$, where $\tbl{\mathcal E^u_t}$ is defined in \eqref{effDomP}, there exists $\rho\in\mathcal C_{tT}$, such that
$$u(\xi,t, T) =\mathbb E_t\left[U(T, \xi\rho)\right].$$
Moreover, \tbl{if} $\xi>0$, $\Pas$, then such a $\rho$ is unique.
\end{Lemma}
\begin{proof}
Let us fix \tbl{$\xi\in\mathcal E^u_t$. }
First, we will show that the set 
\begin{equation}\label{9182}
\mathcal U_t\triangleq\left\{\mathbb E_t\left[U(T,\xi\rho)\right]:~\rho\in\mathcal C_{tT}\right\}
\end{equation}
is closed under pairwise maximization. Let $\rho^1$ and $\rho^2$ be some elements of $\mathcal C_{tT}$ and let $H^1$ and $H^2$ be such that
$$\rho^i \leq 1 + \int_{t+}^TH^i_udR^0_u,\quad i=1,2.$$ We define
$$A :=\left\{\mathbb E_t\left[U(T,\xi\rho^1)\right] > \mathbb E_t\left[U(T,\xi\rho^2)\right]\right\}\in\mathcal F_t,$$
and set $H :=H^11_A + H^21_{A^c}.$ Then we obtain for
$\rho :=\rho^1 1_A + \rho^21_{A^c}$ that
\begin{equation}\label{9181}
\rho \leq 1 + \int_{t+}^TH_udR^0_u,
\end{equation}
as
 \begin{equation}\nonumber
\rho = \left\{
\begin{split}
\rho^1\leq 1 + \int_{t+}^TH^1_udR^0_u,\quad on \quad A,\\
\rho^2\leq 1 + \int_{t+}^TH^2_udR^0_u,\quad on \quad A^c.\\ 
\end{split}
\right.
\end{equation}
Since $1 + \int_{t+}^TH_udS_u\geq 0$, $\Pas$, we conclude from \eqref{9181} that $\rho\in\mathcal C_{tT}$. Consequently,
\begin{displaymath}
\mathbb E_t\left[ U(T,\xi\rho)\right] = \max\left(\mathbb E_t\left[ U(T,\xi\rho^1)\right],   \mathbb E_t\left[ U(T,\xi\rho^2)\right]\right),
\end{displaymath}
i.e., $\mathcal U_t$ defined in \eqref{9182} is closed under pairwise maximization. Applying \cite[Theorem A.3, p. 324]{KSmmf} (or rather an extension of this theorem to extended-real valued random variables, see e.g., \cite[Proposition 2.6.1]{LaurenceBook})
, we deduce that there exists a sequence $(\rho^n)_{n\in\mathbb N}\subset\mathcal C_{tT},$ such that 
\begin{equation}\nonumber
\lim\limits_{n\to\infty}\mathbb E_t\left[ U(T,\xi\rho^n)\right] = 
\esssup\limits_{\rho\in\mathcal C_{tT}}\mathbb E_t\left[ U(T,\xi\rho)\right].
\end{equation}
By Komlos'-type lemma, see e.g., \cite[Lemma A1.1]{DS}, we may find a sequence of convex combinations 
$\tilde \rho^n\in \conv(\rho^n, \rho^{n+1}, \dots)$, 
$n\in\mathbb N$, and a random variable $\hat g$, such that $(\tilde \rho^n_T)_{n\in\mathbb N}$ converges to $\hat g$, \Pas. By concavity of $U(T,\cdot)$, $(\tilde \rho^n_T)_{n\in\mathbb N}$ is also a maximizing sequence in the sense that 
\begin{equation}\label{314}
\lim\limits_{n\to\infty}\mathbb E_t\left[ U(T,\xi\tilde \rho^n)\right] = 
\esssup\limits_{\rho\in\mathcal C_{tT}}\mathbb E_t\left[ U(T,\xi\rho)\right].
\end{equation} 
Similarly to the argument in \cite[Proof of Proposition 3.1]{KS}, \tbl{one  can  show that } $\mathcal C_{tT}$ is closed in probability. Therefore, $\hat g\in\mathcal C_{tT}$. Via \cite[Lemma 3.5]{Mostovyi2015}\footnote{\tbl{Section \ref{resMos15} contains a brief summary of the results from \cite{Mostovyi2015} used in this paper.}} and the symmetry between primal and dual problems in \cite{Mostovyi2015}, one can show that $U^{+}(T,\xi\tilde \rho^n)$, $n\in\mathbb N$, is a uniformly integrable sequence, so is $\mathbb E_t\left[U^{+}(T,\xi\tilde \rho^n)\right]$, $n\in\mathbb N$, and therefore we have
$$\esssup\limits_{\rho\in\mathcal C_{tT}}\mathbb E_t\left[ U(T,\xi \rho)\right] = \lim\limits_{n\to\infty}\mathbb E_t\left[ U(T,\xi\tilde \rho^n)\right]\leq \mathbb E_t\left[ U(T,\xi\hat g)\right],\quad \Pas,$$
and thus $\hat g$ is the maximizer to \eqref{primalProblem}. Further, if $\xi>0$, \Pas, the uniqueness of the maximizer follows from the strict concavity of $U(T,\cdot)$.
\end{proof}

If the \tbl{risky asset} is continuous, in the proof of Lemma \ref{231}, after \eqref{314}, one can apply argument based on a version of the optional decomposition theorem for arbitrary filtrations, see  \cite[Optional Decomposition Theorem 1.4]{KarKar15}.  
\begin{Remark}
If \tbl{the \tbl{risky asset}s} have continuous paths, after \eqref{314}, one can  apply  an argument  based on   a version of the optional  decomposition theorem  for arbitrary filtrations, see \cite[Optional Decomposition Theorem 1.4]{KarKar15}. 
Suppose that  
  $\tilde X^n$, $n\geq 1$,  are nonnegative processes of the form $1 + \int_{t+}^sH^n_udR^0_u$, $s\in(t,T]$, and $\tilde X^n_t = 1$, for some $S$-integrable $H^n$'s, such that $$\tilde X^n_T\geq \tilde \rho^n,\quad \Pas,\quad n\in\mathbb N.$$
As in the proof of \cite[Lemma 4.2]{Mostovyi2015}, we pick a strictly positive $Y\in\mathcal Z_t$ (whose existence follows from Assumption \ref{NUPBR}) and consider $\tilde X^nY$, $n\in\mathbb N$. Let 
$$\mathcal T :=(\mathsf Q\cap (t,T))\cup \{t\}\cup \{T\},$$ where $\mathsf Q$ is the set of rational numbers.

Then, passing to convex combinations, we may find a subsequence of convex combinations ${\tilde{\tilde {X}}^n}$, $n\in\mathbb N$, such that ${\tilde {\tilde {X}}^n}Y$, $n\in\mathbb N$, is Fatou convergent to a supermartingale $VY$ on $\mathcal T$ and such that  
\begin{equation}\label{1241}
V_t\leq\liminf\limits_{s\searrow t, s\in\mathcal T}\liminf\limits_{n\to\infty}\mathbb E\left[ {\tilde{\tilde X}^n_s}\frac{Y_s}{Y_t}|\mathcal F_t\right]\leq 1\quad and \quad V_T\geq \hat g.
\end{equation} 
One  can show that $(V_sZ_s)_{s\in[t,T]}$ is a supermartingale for every $Z\in\mathcal Z_t$ (similarly to \cite[Lemma 4.2]{Mostovyi2015}). 
We stress that the supermartingale property is only required on $[t,T]$. One can see that\footnote{We use the convention $\frac{0}{0} = 0$.} for every supermartingale deflator $Z$, the process $\bar Z$ of the form
$$\bar Z_s = \frac{Z_{t + s}}{Z_t}, ~s\in[0, T-t],$$
is an element of $\mathcal Z_t$. We also set $$\mathcal G_s :=\mathcal F_{t + s},\quad s\in[0, T-t].$$
On the probability space $(\Omega, \mathcal F, \mathbb P)$ endowed with a filtration $(\mathcal G_s)_{s\in[0,T-t]}$,  $(V_s\bar Z_s)_{s\in[0, T-t]}$ is a supermartingale for every $\bar Z\in{\mathcal Z}_t$. Therefore $V$ satisfies the conditions  of (item 1 of) \cite[Optional Decomposition Theorem 1.4]{KarKar15}.  
As a result, there exists a decomposition of $V$ of the form
\begin{equation}\label{1242}
V_s = V_0 + \int_{0+}^s\bar H_ud R^0_u - A_s,\quad s\in(0,T-t],
\end{equation}
where $V_0$ is $\mathcal G_0 = \mathcal F_t$-measurable random variable, $\bar H$ is predictable and $S$-integrable and $A$ is a nondecreasing, right-continuous, adapted process, such that $A_0 = 0$.

We denote $\hat H_{t + s} :=\bar H_{s}$, $s\in[0, T-t]$. 
Therefore, using \eqref{1241} and \eqref{1242}, we deduce that 
$$\hat X_T = 1 + \int_{t+}^T\hat H_{u}dR^0_u \geq\hat g,$$
that is $\hat X_T\in\mathcal C_{tT}$, by the definition of $\mathcal C_{tT}$ in \eqref{defCtT}. 
It follows from Assumption \ref{finiteness} and \cite[Lemma 3.5]{Mostovyi2015} that $\left(U^{+}(T, X_T)\right)_{X\in\mathcal X(x)}$ is uniformly integrable, therefore, so is $\left(\mathbb E_t\left[U^{+}(T, X_T)\right]\right)_{X\in\mathcal X(x)}$, and using the monotonicity of $U(T,\cdot)$, we get 
 $$\esssup\limits_{\rho\in\mathcal C_{tT}}\mathbb E_t\left[ U(T,\xi\rho)\right] = \lim\limits_{n\to\infty}\mathbb E_t\left[ U(T,\xi\tilde \rho^n)\right] \leq \mathbb E_t\left[ U(T,\xi\hat g)\right]\leq
 \mathbb E_t\left[ U(T,\xi\hat X_T)\right].$$
We deduce that $\hat X_T$ is the maximizer to \eqref{primalProblem}. If $\xi>0$, \Pas, the uniqueness of the maximizer follows from the strict concavity of $U$.
\end{Remark}
\subsection{Structure of the dual process}
First, we set 
\begin{displaymath}
V(t,\omega,y) :=\sup\limits_{x>0}(U(t,\omega,x)-xy), \quad (t,\omega,y)\in[0,T]\times\Omega\times[0,\infty).
\end{displaymath}
For $t\in[0,T]$, we set
\begin{equation}\label{defDt}
\mathcal D_t(y):=\{\eta\in m\mathcal F_t:~\eta\leq \tbl{yz_t}~for~some~z\in\mathcal Z_0\},\quad y\geq 0.
\end{equation}
i.e., $\mathcal D_t(y)$ is a subset \tbl{of} the closure of the convex solid hull of the elements of $\tbl{y\mathcal Z_0}$ sampled at time $t$.
We define
\begin{equation}\label{defNt}
\mathcal N_t :=\bigcup\limits_{y\geq 0}\mathcal D_t(y).
\end{equation}
For $t\in[0,T]$, let $\mathcal Z_t$ be given by \eqref{defZt} and we set
\begin{equation}\label{dualProblem}
v(\eta, t, T) :=\essinf\limits_{z\in\mathcal Z_t}\mathbb E_t\left[V(T,\eta z_T) \right],\quad \eta \in \mathcal N_t.
\end{equation}

\begin{Lemma}\label{lemDualExistence}
Under the conditions of Lemma \ref{231},  for every $\eta \in\mathcal N_t$, there exists $\widehat z\in\mathcal Z_t$, such that 
\begin{equation}\label{eqnDualRep}
v(\eta, t, T) = \mathbb E\left[V(T,\eta \widehat z_T)|\mathcal F_t\right].
\end{equation}
\end{Lemma}
\begin{proof}
 Let us consider $\bar z^1$ and $\bar z^2$ in $\mathcal Z_t$, let $$A :=\{\omega:~ \mathbb E\left[V(T,\eta \bar z^1_T)|\mathcal F_t\right](\omega)<\mathbb E\left[V(T,\eta \bar z^2_T)|\mathcal F_t\right](\omega)\}\in\mathcal F_t$$ and 
$$\bar z_{t'}:=\bar z^1_{t'}1_A + \bar z^2_{t'}1_{A^c},\quad t'\in[t,T].$$
Then on $[t, T]$, $\tbl{\bar z}X$ is a supermartingale deflator for every $X\in\mathcal X(1)$. Therefore, $\bar z\in\mathcal Z_t.$ By direct computations, we have
$$\mathbb E\left[V(T,\eta \bar z_T)|\mathcal F_t\right] = \min\left(\mathbb E\left[V(T,\eta \bar z^1_T)|\mathcal F_t\right], \mathbb E\left[V(T,\eta \bar z^2_T)|\mathcal F_t\right] \right).$$
Therefore, 
from \cite[Proposition 2.6.1]{LaurenceBook}
, we deduce that  there exists a sequence $(z^n)_{n\in\mathbb N}$, such that
\begin{equation}\label{9184}
\lim\limits_{n\to\infty} \mathbb E\left[V(T,\eta z^n_T)|\mathcal F_t \right] = v(\eta, t, T),\quad \mathbb P-a.s.
\end{equation}
By passing to convex combinations, we obtain a subsequence, which we do not relabel, such that \textcolor{black}{$\lim\limits_{n\to\infty}z^n = \widehat z$, where the limit is considered in the Fatou sense (in the terminology of \cite[Definition 5.2]{FK}) on the set of rational numbers on $(t,T)$ augmented with $t$ and $T$}. Note that by convexity of $V(T,\cdot)$, such a subsequence will also satisfy \eqref{9184}.  It follows from the definition of $\mathcal N_t$ and \cite[Lemma 3.5]{Mostovyi2015} that $\left(V^{-}(T,\eta z^n_T)\right)_{n\in\mathbb N}$ is uniformly integrable. Therefore, we get 
\begin{equation}\label{9185}
\lim\limits_{n\to\infty}\mathbb E\left[ V\left(T,\eta z^n_T \right)|\mathcal F_t\right]\geq
\mathbb E\left[ V\left(T,\eta \widehat z_T\right)|\mathcal F_t\right].
\end{equation}
By direct computations, 
 we deduce that for every $X\in\mathcal X(1)$, $(\widehat z_{t'}X_{t'})_{t'\in[t,T]}$ is a nonnegative supermartingale on $[t,T]$, such that $\tbl{\widehat z_t}\leq 1$ by properties of Fatou-convergence. Therefore
$\widehat z\in\mathcal Z_t$. Via \eqref{9184} and \eqref{9185}, we conclude that \eqref{eqnDualRep}  holds.
\end{proof}
Let 
\begin{equation}\label{defBt+}
\mathcal B^+_t :=\left\{X\in m\mathcal F_t:~X\in[0,1],~\Pas\right\}.
\end{equation}
and
\begin{equation}\label{dualEffdom}
\tbl{\mathcal E^v_t} := \left\{\eta\in\mathcal N_t: ~there~exists~z\in\mathcal Z_{t}, ~such~that~\mathbb E\left[V^+(T,\eta z_T)\right]<\infty\right\},
\end{equation}
which corresponds to the effective domain of $v(\cdot, t,T)$. 
\begin{Lemma}\label{lemNconvex}
Let the condition of Lemma \ref{231} hold, $\lambda\in\mathcal B^{+}_t$ and $\eta^1$ and $\eta^2$ are some elements of $\tbl{\mathcal E^v_t}$, then 
$\eta := \lambda \eta^1+(1 - \lambda)\eta^2\in\tbl{\mathcal E^v_t}$ and we have
\begin{equation}
\label{2281}
v(\eta,t, T)\leq \lambda v(\eta^1, t, T)+ (1-\lambda)v(\eta^2,t, T).
\end{equation}
\end{Lemma}
\begin{proof}
As $\eta^i\in\mathcal D_t(y^i)$, for some $y^i>0$, $i = 1,2$, one can see that $\eta\leq Y_t$ for some $Y\in{(y^1 + y^2)\mathcal Z_0}$. Next we will show \eqref{2281}. 
By Lemma \ref{lemDualExistence}, we deduce the existence of $\hat z^1$ and $\hat z^2$, the optimizers to \eqref{dualProblem} corresponding to $\eta^1$ and $\eta^2$, respectively. 
With \begin{equation}\nonumber
\begin{split}
z_s:=&\frac{\lambda \eta^1\hat z^1_s + (1-\lambda)\eta^2 \hat z^2_s}{\eta}{1_{\{\eta\neq 0\}}+\left(\lambda \hat z^1_s + (1-\lambda)\hat z^2_s\right)1_{\{\eta = 0\}}},\quad s\in[t,T],
\end{split}
\end{equation}
by direct computations, one can see that 
$(z_{t'}X_{t'})_{t'\in[t,T]}$ is a supermartingale for every $X\in\mathcal X(1)$. By construction, $z_t \leq 1$. We conclude that $z\in\mathcal Z_t$. 
\tbl{To show that $\eta\in\tbl{\mathcal E^v_t}$, first  we observe that on $\{\eta = 0\}$, we have  $\lambda \hat  \eta^1 = (1-\lambda)\eta^2 = 0$. Therefore, we obtain  
\begin{displaymath}
\begin{split}
\eta z_T  &= \eta z_T 1_{\{\eta >0\}} + 0 \cdot  z_T1_{\{\eta = 0\}} \\
&= \left(\lambda \eta^1\hat z^1_T + (1-\lambda)\eta^2\hat z^2_T\right)1_{\{\eta >0\}} + 0\cdot \left(\lambda \hat z^1_s + (1-\lambda)\hat z^2_s\right) 1_{\{\eta = 0\}} \\
&= \left(\lambda \eta^1\hat z^1_T + (1-\lambda)\eta^2\hat z^2_T\right)1_{\{\eta >0\}} +  \left(\lambda \eta^1\hat z^1_s + (1-\lambda)\eta^2\hat z^2_s\right) 1_{\{\eta = 0\}} \\
&= \lambda \eta^1\hat z^1_T + (1-\lambda)\eta^2\hat z^2_T.
\end{split}
\end{displaymath}
Therefore,  using the convexity of $V^{+}(T,\cdot)$, we  get
\begin{displaymath}
\begin{split}
\mathbb  E_t\left[ V^+(T, \eta  z_T)\right] &= \mathbb  E_t\left[ V^+(T, \lambda \eta^1\hat z^1_T + (1-\lambda)\eta^2\hat z^2_T)\right]   \\
&\leq 
 \lambda\mathbb  E_t\left[ V^+(T,  \eta^1\hat z^1_T)\right] + 
  (1-\lambda)\mathbb  E_t\left[ V^+(T, \eta^2\hat z^2_T)\right].
  \end{split}
\end{displaymath}
Therefore, as $\eta^1$  and $\eta^2$  are in $\mathcal E^v_t$, we deduce  that so  is  $\eta$. 
} 
Likewise, using the convexity of $V$, we get
\begin{equation}\label{421}
\begin{split}
&v(\eta, t\tbl{, T}) \leq \mathbb E_t\left[V\left(T, \eta z_T\right)\right]  = \mathbb E_t\left[ V\left(T,\lambda \eta^1\hat z^1_T + (1-\lambda)\eta^2\hat z^2_T\right)\right] \\
&\leq \lambda\mathbb E_t\left[ V\left(T,\eta^1\hat z^1_T\right)\right] + (1-\lambda)\mathbb E_t\left[ V\left(T,\eta^2\hat z^2_T\right)\right]  = \lambda v(\eta^1,t, T) + (1-\lambda)v(\eta^2,t, T).
\end{split}
\end{equation}
Thus, \eqref{2281} holds. 

\end{proof}

An important role in the proofs below will be played by the set
\begin{equation}\label{defGt}
\mathcal G_t :=\left\{\eta z_T:~\eta \in \mathcal N_t, z\in\mathcal Z_t \right\},
\end{equation}
\tbl{
which is characterized in the Lemma \ref{lemG} below. The proof of Lemma \ref{lemG} is straightforward and is omitted.}
\begin{Lemma}\label{lemG}
Let the condition of Lemma \ref{231} hold. \tbl{Then the set $\mathcal G_t$ is closed under convex concatenations in the following sense: for a weight $\lambda\in\mathcal B^{+}_t$,  and $\eta^iz^i\in\mathcal G_t$, $i = 1,2$ we set}
\begin{equation}\label{8231}
\eta :=\lambda \eta^1 +(1-\lambda)\eta^2.
\end{equation}
\tbl{Then $\eta\in\mathcal N_t$,} and we have: 
\begin{equation}\label{311}
\begin{split}
z_s:=&\frac{\lambda \eta^1z^1_s + (1-\lambda)\eta^2 z^2_s}{\eta}{1_{\{\eta\neq 0\}}+\left(\lambda z^1_s + (1-\lambda)z^2_s\right)1_{\{\eta = 0\}}},\quad s\in[t,T],\\
&is~an~element~of~\mathcal Z_t,\\
\end{split}
\end{equation}
\begin{equation}\label{312}
\lambda \eta^1z^1_T + (1 - \lambda)\eta^2z^2_T =\eta z_T\in\mathcal G_t.
\end{equation}
In particular, \eqref{8231} holds if $\lambda$ is a constant taking values in $[0,1]$, i.e., $\mathcal G_t$ is a convex set and \tbl{if} $\eta^i$'s are in $\tbl{\mathcal E^v_t}$, then $\eta\in\tbl{\mathcal E^v_t}$.
\end{Lemma}
\begin{proof} Let us consider $\eta$ and $z$ are defined in \eqref{8231} and \eqref{311}, respectively.
As $\eta\leq \eta^1 + \eta^2$, we deduce (trivially) that $\eta\in\mathcal N_t$.
It follows from Lemma \ref{lemNconvex} that if, additionally, $\eta^1$ and $\eta^2$ are some elements of $\tbl{\mathcal E^v_t}$, then $\eta\in\tbl{\mathcal E^v_t}$. Following the proof of the same lemma, $z\in\mathcal Z_t$
The validity of $$\lambda \eta^1z^1_T + (1 - \lambda)\eta^2z^2_T =\eta z_T$$ is a  consequence of the definitions of $\eta$ and $z$. As, by the argument above, $\eta\in\mathcal N_t$ and $z\in\mathcal Z_t$, we deduce that \eqref{312} holds. Thus, in particular, $\mathcal G_t$ is convex. 
\end{proof}
\subsection{Conjugacy of $u$ and $v$}
We recall that $\mathcal C(x)$, $x>0$, are defined in \eqref{defCt}. The goal is to show that 
$$u(\xi, t, T) = \essinf\limits_{\eta \in \mathcal N_t}\left( v(\eta, t, T) + \xi\eta\right),\quad \xi\in \bigcup\limits_{x>0}\mathcal C(x)$$and
$$v(\eta, t, T) = \esssup\limits_{\xi\in\bigcup\limits_{x>0}\mathcal C_t(x)}\left( u(\xi, t, T) - \xi \eta \right),\quad \eta \in\mathcal N_t.$$
\begin{Lemma}\label{lemSubConjugacy} 
Let the condition of Lemma \ref{231} hold. Then  for every $\eta\in\mathcal N_t$ and $\xi\in\bigcup\limits_{x\geq 0}\mathcal C_t(x)$, we have:
$$u(\xi,t, T) \leq v(\eta, t, T) + \xi\eta.$$
\end{Lemma}
\begin{proof}
Let $\xi\in\bigcup\limits_{x\geq 0}\mathcal C_{t}(x)$ and $\eta\in\mathcal N_t$ be fixed. We need to show that
\begin{equation}\label{6141}\esssup\limits_{\rho\in\mathcal C_{tT}}\mathbb E\left[ U\left(T,\xi\rho\right)|\mathcal F_t\right] -\xi\eta \leq \essinf\limits_{z\in\mathcal Z_t}\mathbb E\left[V\left(T,\eta z_T\right)|\mathcal F_t \right].
\end{equation}
From the definition of the conjugate function, for every $\rho\in\mathcal C_{tT}$ and every $z\in\mathcal Z_t$, we have
\begin{equation}\label{6142}
\begin{split}
\mathbb E_t\left[ U\left(T,\xi\rho\right)\right]&\leq  \mathbb E_t\left[V\left(T,\eta z_T\right) + \xi\rho\eta z_T\right] 
.\\
\end{split}
\end{equation}
As $z\in\mathcal Z_t$, $\mathbb E_t\left[\rho{z_T}\right]\leq 1$ and we have
\begin{equation}\label{8241}
\mathbb E_t\left[\xi\rho\eta z_T\right] = \xi{\eta}\mathbb E_t\left[\rho{z_T}\right]\leq \xi{\eta}.
\end{equation}
Combining  \eqref{6142} and \eqref{8241}, we get
$$\mathbb E_t\left[ U\left(T,\xi\rho\right)\right] \leq 
\mathbb E_t\left[V\left(T,\eta{z_T}\right)  \right]
+{\eta}\xi,
$$
which implies
\eqref{6141}. This completes the proof of the lemma.
\end{proof}
\begin{Lemma}\label{996} Let the condition of Lemma \ref{231} hold and $\eta \in\mathcal N_t$. Then we have 
\textcolor{black}{
\begin{equation}\label{8282}
\lim\limits_{x\to \infty}\essinf\limits_{z\in\mathcal Z_t}
\esssup\limits_{\xi\in\mathcal C_T(x)}
\mathbb E_t\left[\left(U(T,\xi) - \xi\eta z_T\right)\right] \geq \essinf\limits_{z\in\mathcal Z_t}\mathbb E_t\left[V(T,\eta z_T) \right] = v(\eta, t,T).
\end{equation}
Further,} for every $A\in\mathcal F_t$, we have
\begin{equation}\label{9232}
\lim\limits_{x\to \infty}\inf\limits_{z\in\mathcal Z_t}
\sup\limits_{\xi\in\mathcal C_T(x)}
\mathbb E\left[\left(U(T,\xi) - \xi\eta z_T\right)1_A\right] \geq \inf\limits_{z\in\mathcal Z_t}\mathbb E\left[V(T,\eta z_T)1_A \right].
\end{equation}
\end{Lemma}
\begin{proof}
{\it Step 1.} First, we suppose that 
\begin{equation}\label{1011}
\mathbb E\left[U(T,1)\right]>-\infty.
\end{equation}
Let us  set
$$V^n(T,y) :=\sup\limits_{x\in(0,n]}(U(T,x) - xy),\quad y\geq 0,\quad n\in\mathbb N.$$
\tbl{Next, we fix $z\in\mathcal Z_t$.} Then, for every 
$n\in\mathbb N$, as $n\in\mathcal C_T(n)$, we get
\begin{equation}\label{8284}
\esssup\limits_{\xi\in\mathcal C_T(n)}
\mathbb E_t\left[\left(U(T,\xi) - \xi\eta z_T\right)\right] \geq 
\mathbb E_t\left[V^n\left(T,\eta z_T\right)\right].
\end{equation}
Next we set $$v^n(\eta, t,T) :=\essinf\limits_{z\in\mathcal Z_t} \mathbb E_t\left[V^n\left(T,\eta z_T\right)\right],\quad \eta \in\mathcal N_t,$$
and observe that $v^n(\eta, t,T)$, $n\in\mathbb N$, in an increasing sequence. 
From \eqref{8284}, we obtain 
\begin{equation}\label{993}
\lim\limits_{n\to\infty}\essinf\limits_{z\in\mathcal Z_t}\esssup\limits_{\xi\in\mathcal C_T(n)}
\mathbb E_t\left[\left(U(T,\xi) - \xi\eta z_T\right)\right] \geq \lim\limits_{n\to\infty}v^n(\eta, t,T).
\end{equation}
As $V^n(T,y)\leq V(T,y)$, $y\geq 0$, similarly to Lemma \ref{lemDualExistence}, we deduce that  there exists $\hat z^n\in\mathcal Z_t$, such that 
$$v^n(\eta, t,T) = \mathbb E_t\left[V^n\left(T,\eta \hat z^n_T\right) \right],\quad n\in\mathbb N.$$
One can pass to convex combinations, which we denote $\tilde z^n$, $n\in\mathbb N$, to obtain a  Fatou-limit of $\tilde z^n$'s, which we denote $\hat z\in\mathcal Z_t$. 
As $$V^n(T,y)\geq V^2(T,y) \geq V(T,y)1_{\{y\geq U'(T,2)\}} + (U(T,2) - 2U'(T,2))1_{\{y<U'(T,2)\}}, \quad y\geq 0,n\geq 2,$$
using convexity of $U(T, \cdot)$, we get $U'(T, 2)\leq U(T,2) - U(T,1)$, and thus 
$$V^n(T,y)\geq \min\left( V(T,y),2U(T,1) - U(T,2)\right),\quad y\geq 0,n\geq 2,$$
and (similarly to \cite[Lemma 3.9]{Mostovyi2015}) uniform integrability of $(V^n)^{-}\left(T,\tilde z^n_T\eta\right)$, $n\geq 2$, follows from \cite[Lemma 3.5]{Mostovyi2015} and \eqref{1011}.  
As a consequence, using convexity of $V^n(T,\cdot)$'s and Fatou's lemma, we get 
\begin{equation}\label{994}
\lim\limits_{n\to\infty}v^n(\eta, t,T) = \lim\limits_{n\to\infty}\mathbb E_t\left[V^n(T,\hat z^n_T \eta) \right]\geq 
\lim\limits_{n\to\infty}\mathbb E_t\left[V^n(T,\tilde z^n_T \eta) \right]\geq
\mathbb E_t\left[V(T,\hat z_T \eta) \right]\geq
v(\eta, t,T).
\end{equation}
 Combining \eqref{993} and \eqref{994}, we obtain \eqref{8282}. In turn,  \eqref{9232} can be proven similarly. 
 
 {\it Step 2.} \tbl{Here we do not suppose that \eqref{1011} holds.} This case can be reduced to the one above by taking $\hat \rho = \argmax\limits_{\xi\in\mathcal C_{tT}}\mathbb E_t\left[ U\left(T,\tfrac{1}{2}\xi\rho\right)\right]$ and by setting
 $\rho = \max\left( \hat \rho, \tfrac{1}{2}\right).$ Then $0<\rho\in\mathcal C_{tT}$ and for 
 \begin{displaymath}
 \begin{array}{c}
 \tilde U(T,x) :=U(T,\rho x),\quad \tilde V(T,y) :=V\left(T,\frac{y}{\rho}\right),\\
 \tilde C_{tT} :=\left\{\tilde \rho:~\tilde \rho \rho\in\mathcal C_{tT} \right\},\quad \tilde {\mathcal Z}_t:=\left\{\tilde z:~\frac{z}{\rho}\in \mathcal Z_t\right\}.\\
 \end{array}
 \end{displaymath} 
 Then we can represent $u$ and $v$ as
 \begin{displaymath}
 u(\xi, t, T) = \esssup\limits_{\tilde \rho \in\mathcal C_{tT}}\mathbb E_t\left[\tilde U(T,\tilde \rho) \right],\quad v(\eta, t,T) = \essinf\limits_{\tilde z\in\mathcal Z_t}\mathbb E_t\left[\tilde V(T,\eta \tilde z)\right]
 \end{displaymath}
 and $\tilde U$ satisfies \eqref{1011}. Then, \eqref{8282} and  \eqref{9232} follow from Step 1. 
%
\end{proof}
Let us define 
\begin{equation}\label{defC'}\mathcal C'_{tT} :=\left\{\rho\in\mathcal C_{tT}:~\esssup\limits_{z\in\mathcal Z_t}\mathbb E_t[\rho z_T] = 1\right\}
\end{equation}
and adapt to our settings the notation from \cite{Delbaen-Schachermayer1997}.
\begin{equation}\label{defKmax}
\mathcal K^{\max}_{tT} :=\left\{maximal~elements~of~\mathcal X_{tT}(1)\right\}.
\end{equation}
The following lemma extends some characterization of maximal 1-admissible contingent claims from \cite{Delbaen-Schachermayer1997} to the present settings, mainly to Assumption \ref{NUPBR}.
\begin{Lemma}\label{985}
Let Assumption \ref{NUPBR} holds. Then we have
\begin{equation}\label{9811}\mathcal K^{\max}_{tT}\subseteq\mathcal C'_{tT}\subset \mathcal C_{tT}.
\end{equation} 
As a consequence, for every $x>0$ and $\xi\in\mathcal C_t(x)$, we have 
\begin{equation}\label{92213} u(\xi, t,T) = \esssup\limits_{\phi\in\mathcal C'_{tT}}\mathbb E_t\left[U(T,\xi\phi)\right]= \esssup\limits_{\phi\in\mathcal C_{tT}}\mathbb E_t\left[U(T,\xi\phi)\right],\end{equation}
and for every $A\in\mathcal F_t$, we have
\begin{equation}\label{92212}
\sup\limits_{\phi\in\mathcal C'_{tT}}\mathbb E\left[U(T,\xi\phi)1_A\right]= \sup\limits_{\phi\in\mathcal C_{tT}}\mathbb E\left[U(T,\xi\phi)1_A\right].
\end{equation}
\end{Lemma}
\begin{proof}
The proof if based on a change of num\'eraire idea. By the results of \cite{KarKar07}, Assumption \ref{NUPBR}, implies the existence of the num\'eraire portfolio $N$. Let us assume that for some trading strategy $G$, $X^{tT}:=1 + \int_{t+}^T G_sdR^0_s$ is a (maximal) 
element of $\mathcal K^{\max}_{tT}$, then, as we can extend $G$ by $0$ on $[0,t]$ to obtain an element of $K^{\max}$, one can see that $\frac{X^{tT}}{N}$ is a maximal element under the num\'eraire $N$, and NFLVR holds for $\left(\frac{1}{N},\frac{R^0}{N}\right)$. As  densities of locally equivalent martingale measures under the new num\'eraire can be represented as $z'zN$, where $z'$ is a supermartingale deflator for $S$ on $[0,t]$ and $z$ is an element of $\mathcal Z_t$, we deduce from \cite[Theorem 2.5]{Delbaen-Schachermayer1997}, the existence of $z'z$, such that 
$$ 1= \mathbb E\left[z'_tz_TN_T \frac{X^{tT}_T}{N_T} \right]=\mathbb E\left[z'_tz_T{X^{tT}_T} \right] = \mathbb E\left[z'_t \mathbb E_t\left[z_T{X^{tT}_T}\right] \right].$$
As $\mathbb E_t[z_T X^{tT}_T]\leq 1$, by construction, it follows that 
 $\mathbb E_t\left[z_T{X^{tT}_T}\right] =1$, \Pas, and therefore we have
\begin{equation}\label{9810}
K^{\max}_{tT}\subseteq \mathcal C'_{tT}.
\end{equation} Similarly to \cite[Proposition 3.1]{KS}, we deduce that
$$\mathcal C_{tT} = \left\{c\in m\mathcal F_T:~c\leq h,~for~some~h\in \mathcal K^{\max}_{tT}\right\} = \left\{c\in m\mathcal F_T:~\esssup\limits_{z\in\mathcal Z_t}\mathbb E_t[cz_T]\leq 1 \right\},$$
and thus \eqref{9811} holds. Moreover, as $U(T,\cdot)$ is nondecreasing, we get
$$
\esssup\limits_{\phi\in\mathcal K^{\max}_{tT}}\mathbb E_t\left[U(T,\xi\phi)\right]= \esssup\limits_{\phi\in\mathcal C_{tT}}\mathbb E_t\left[U(T,\xi\phi)\right].$$
Combining the latter equality with \eqref{9811}, we obtain \eqref{92213}. Finally, \eqref{92212} can be obtained similarly to \eqref{92213}.
\end{proof}

\begin{Lemma}\label{lem9221}
Let the condition of Lemma \ref{231} hold and $\eta \in\mathcal N_t$ be fixed. Then for every $A\in\mathcal F_t$, we have
\begin{displaymath}
\begin{split}
\inf\limits_{z\in\mathcal Z_t}\mathbb E\left[V(T,\eta z_T)1_{A}\right]
&=
\lim\limits_{x\to\infty}\inf\limits_{z\in\mathcal Z_t}\sup\limits_{\xi\rho\in\mathcal C_T(x)}
\mathbb E\left[\left(U(T,\xi\rho) - \eta\xi\rho z_T\right)1_A\right] \\
&=
\sup\limits_{\xi\rho\in\bigcup\limits_{x>0}\mathcal C_T(x)}
\inf\limits_{z\in\mathcal Z_t}
\mathbb E\left[\left(U(T,\xi\rho) - \eta\xi\rho z_T\right)1_A\right].\\
\end{split}
\end{displaymath}
\end{Lemma}
\begin{proof}
The first equality follows from Lemma \ref{996} (see \eqref{9232}) and the definition of $V$, whereas the second one is a consequence if the minimax theorem, see  \cite[Theorem B.3]{bankKauppila}. 
\end{proof}
\begin{Remark}
The proof of Lemma \ref{lem9221} follows the structure of \cite[Lemma 3.9]{Mostovyi2015}. However, in view of \cite[Theorem B.3]{bankKauppila}, one does not need to truncate the domain of $u$ and to invoke the Banach-Alaoglu's theorem.
\end{Remark}
\begin{Remark}[On the multiplicative decomposition of $\mathcal C_{tT}$]\label{rem9231}
 Let us consider 
$$\alpha (\rho) :=\esssup\limits_{z\in\mathcal Z_t}\mathbb E_t\left[ \rho z_T\right],\quad \rho\in\mathcal C_{tT}.$$
Then, for every $\rho\in\mathcal C_{tT}$, $\alpha(\rho)$ takes values in $[0,1]$, and it follows from Assumption \ref{NUPBR} that  
\begin{equation}\label{983}
\mathbb P\left[\left\{\rho>0\right\} \cap \left\{\alpha(\rho) = 0 \right\}\right] = 0.
\end{equation}
 and we recall that $\mathcal C'_{tT}$ and $\mathcal B^+_t$ are defined in \eqref{defC'} and \eqref{defBt+}, respectively. Then using \eqref{983}, we get
$$\rho = 1_{\{\alpha(\rho) > 0\}}\alpha(\rho)\frac{\rho}{\alpha(\rho)} + 1_{\{\alpha(\rho) = 0\}}\rho = 1_{\{\alpha(\rho)> 0\}}\alpha(\rho)\frac{\rho}{\alpha(\rho)},$$
i.e., a multiplicative decomposition of $\rho$ into an element of $\mathcal B^{+}_t$ and an element of $\mathcal C'_{tT}$, which holds on  $\{\alpha(\rho) > 0\}$, and which we can extend to $\{\alpha(\rho) = 0\}$ by $\alpha(\rho)$ multiplied by any element of $\mathcal C'_{tT}$ (restricted to $\{\alpha(\rho) = 0\}$).
\end{Remark}
\begin{Lemma}\label{lem9222}
Let the condition of Lemma \ref{231} hold. Then we have
\begin{equation}\label{9226}
\sup\limits_{\xi\in\mathcal C_t(x)}\sup\limits_{\rho\in\mathcal C_{tT}}\mathbb E\left[\left(U(T,\xi\rho) - \xi\eta\right)1_{A} \right]=
\sup\limits_{\xi\in\mathcal C_t(x)}\sup\limits_{\rho\in\mathcal C_{tT}}\inf\limits_{z\in\mathcal Z_t}\mathbb E\left[\left(U(T,\xi\rho) - \xi\eta \rho z_T\right)1_{A} \right].
\end{equation}
\end{Lemma}
\begin{proof}
With $\mathcal B^{+}_t$ and $\mathcal C'_{tT}$ defined in \eqref{defC'} and \eqref{defBt+}, respectively, and following the argument of Remark \ref{rem9231}, we can rewrite the right-hand side of \eqref{9226} as
\begin{equation}\label{92211}
\begin{split}
&\sup\limits_{\xi\in\mathcal C_t(x)}\sup\limits_{\rho\in\mathcal C_{tT}}\inf\limits_{z\in\mathcal Z_t}\mathbb E\left[\left(U(T,\xi\rho) - \xi\eta \rho z_T\right)1_{A} \right] \\
=&\sup\limits_{\xi\in\mathcal C_t(x)}\sup\limits_{\rho\in\mathcal C_{tT}}\left(\mathbb E\left[\left(U(T,\xi\rho\right)1_A \right]- \sup\limits_{z\in\mathcal Z_t}\mathbb E\left[\xi\eta \rho z_T1_{A} \right]\right)\\
=&\sup\limits_{\xi\in\mathcal C_t(x)}\sup\limits_{\alpha\in\mathcal B^{+}_t}\sup\limits_{\phi\in\mathcal C'_{tT}}\left(\mathbb E\left[\left(U(T,\xi\alpha\phi\right)1_A \right]- \sup\limits_{z\in\mathcal Z_t}\mathbb E\left[\xi\eta \alpha\phi z_T1_{A} \right]\right).\\ 
\end{split}
\end{equation}
Let us consider the latter term, $\sup\limits_{z\in\mathcal Z_t}\mathbb E\left[\xi\eta \alpha\phi z_T1_{A} \right]$, which we can rewrite as
\begin{equation}\label{9227}
\sup\limits_{z\in\mathcal Z_t}\mathbb E\left[\xi\eta \alpha\phi z_T1_{A} \right]
=\sup\limits_{z\in\mathcal Z_t}\mathbb E\left[\xi\eta \alpha1_{A}\mathbb E_t\left[\phi z_T\right] \right],
\end{equation}
where by the respective definitions of $\mathcal C_t(x)$, $\mathcal N_t$, and $\mathcal B^{+}_t$, we deduce that 
\begin{equation}\label{9229}
0\leq\mathbb E\left[\xi\eta \alpha1_{A}\right]<\infty,
\end{equation}
and from the definition of $\mathcal C'_{tT}$, for every $\phi\in\mathcal C'_{tT}$, we have 
$$\esssup\limits_{z\in\mathcal Z_t}\mathbb E_t\left[\phi z_T\right] = 1,\quad \Pas$$
Therefore, for every $z\in\mathcal Z_t$, we obtain
\begin{equation}\nonumber
\begin{array}{c}
\mathbb E\left[\xi\eta \alpha1_{A}\mathbb E_t\left[\phi z_T\right] \right] \leq \mathbb E\left[\xi\eta \alpha1_{A}\esssup\limits_{z\in\mathcal Z_t}\mathbb E_t\left[\phi z_T\right] \right]  
= \mathbb E\left[\xi\eta \alpha1_{A}\right].\\
\end{array}
\end{equation}
Consequently, for every $\xi\in\mathcal C_t(x)$, $\eta\in\mathcal N_t$, $\alpha\in\mathcal B^{+}_t$, $\phi\in\mathcal C'_{tT}$, we get
\begin{equation}\label{9228}
\sup\limits_{z\in\mathcal Z_t}\mathbb E\left[\xi\eta \alpha1_{A}\mathbb E_t\left[\phi z_T\right] \right] \leq \mathbb E\left[\xi\eta \alpha1_{A}\right].
\end{equation}
On the other hand, let us fix $\phi\in\mathcal C'_{tT}$ and two arbitrary elements of $\mathcal Z_t$, $\bar z^1$ and $\bar z^2$. With
$$B:=\left\{ \mathbb E_t\left[ \phi\bar z^1_T\right]> \mathbb E_t\left[ \phi\bar z^2_T\right]\right\},$$
one can see that $$\bar z:=1_B\bar z^1 + 1_{B^c}\bar z^2\in\mathcal Z_t,$$ is such that 
$$\mathbb E_t\left[ \phi\bar z_T\right] = \max\left(\mathbb E_t\left[ \phi\bar z^1_T\right], \mathbb E_t\left[ \phi\bar z^2_T\right]\right),$$
and therefore by \cite[Theorem A.2.3, p. 215]{Pham09}, we deduce that there exists a sequence $(z^n)_{n\in\mathbb N}\subset \mathcal Z_t$, such that
\begin{equation}\label{92210}
\lim\limits_{n\to\infty}\mathbb E_t\left[ \phi z^n_T\right] = \esssup\limits_{z\in\mathcal Z_t}\mathbb E_t\left[ \phi z_T\right] = 1,\quad \Pas,
\end{equation}
where the last equality follows from the definition of $\mathcal C'_{tT}$.
 Therefore, the left-hand side in \eqref{9227}, can be bounded from below as follows.
\begin{equation}\label{92281}
\sup\limits_{z\in\mathcal Z_t}\mathbb E\left[\xi\eta \alpha1_{A}\mathbb E_t\left[\phi z_T\right] \right]\geq 
\lim\limits_{n\to\infty}
\mathbb E\left[\xi\eta \alpha1_{A}\mathbb E_t\left[\phi z^n_T\right] \right]
\end{equation}
As $\mathbb E_t\left[\phi z^n_T\right] \leq 1$, $n\in\mathbb N$, \Pas,  and in view of \eqref{9229}, an application of the dominated convergence theorem in the right-hand side of \eqref{92281} gives 
\begin{displaymath}
\begin{split}
\lim\limits_{n\to\infty}
\mathbb E\left[\xi\eta \alpha1_{A}\mathbb E_t\left[\phi z^n_T\right] \right] = 
\mathbb E&\left[\xi\eta \alpha1_{A}\lim\limits_{n\to\infty}\mathbb E_t\left[\phi z^n_T\right] \right]\\ 
 &=\mathbb E\left[\xi\eta \alpha1_{A}\esssup\limits_{z\in\mathcal Z_t}\mathbb E_t\left[\phi z^n_T\right] \right] = \mathbb E\left[\xi\eta \alpha1_{A}\right],
\end{split}
\end{displaymath}
where we used \eqref{92210} in the last equality. Combining these equalities with \eqref{92281}, we get
$$\sup\limits_{z\in\mathcal Z_t}\mathbb E\left[\xi\eta \alpha1_{A}\mathbb E_t\left[\phi z_T\right] \right]\geq \mathbb E\left[\xi\eta \alpha1_{A}\right],$$
which together with \eqref{9228} imply that 
$$\sup\limits_{z\in\mathcal Z_t}\mathbb E\left[\xi\eta \alpha1_{A}\mathbb E_t\left[\phi z_T\right] \right]= \mathbb E\left[\xi\eta \alpha1_{A}\right].$$
Plugging this equality into \eqref{92211}, we obtain 
\begin{equation}\label{92215}
\begin{split}
&\sup\limits_{\xi\in\mathcal C_t(x)}\sup\limits_{\alpha\in\mathcal B^{+}_t}\sup\limits_{\phi\in\mathcal C'_{tT}}\left(\mathbb E\left[\left(U(T, \xi\alpha\phi\right)1_A \right]- \sup\limits_{z\in\mathcal Z_t}\mathbb E\left[\xi\eta \alpha\phi z_T1_{A} \right]\right) \\
=&
\sup\limits_{\xi\in\mathcal C_t(x)}\sup\limits_{\alpha\in\mathcal B^{+}_t}\sup\limits_{\phi\in\mathcal C'_{tT}}\left(\mathbb E\left[\left(U(T, \xi\alpha\phi\right)1_A \right]- \mathbb E\left[\xi\eta \alpha1_{A} \right]\right).\\ 
\end{split}
\end{equation}
 \tbl{Note that $ \mathcal C_t(x) = \mathcal B^{+}_t \mathcal C_t(x)$, that is, for every $\alpha\in\mathcal B^{+}_t $ and $\xi \in\mathcal C_t(x)$, we have $\alpha \xi\in\mathcal C_t(x)$}. It follows from Lemma \ref{985} (see \eqref{92212}), that in \eqref{92215} the latter equality can be rewritten as
\begin{displaymath}\begin{split}
\sup\limits_{\xi\in\mathcal C_t(x)}\sup\limits_{\alpha\in\mathcal B^{+}_t}&\left(\sup\limits_{\phi\in\mathcal C'_{tT}}\mathbb E\left[\left(U(T,\xi\alpha\phi\right)1_A \right]- \mathbb E\left[\xi\eta \alpha1_{A} \right]\right) 
\\
&= \sup\limits_{\xi\in\mathcal C_t(x)}\left(\sup\limits_{\phi\in\mathcal C_{tT}}\mathbb E\left[\left(U(T,\xi\phi\right)1_A \right]- \mathbb E\left[\xi\eta 1_{A} \right]\right).
\end{split}\end{displaymath}
Finally, combining the latter equality with (chains of equalities) \eqref{92211} and \eqref{92215}, we conclude that 
\textcolor{black}{$$\sup\limits_{\xi\in\mathcal C_t(x)}\sup\limits_{\rho\in\mathcal C_{tT}}\inf\limits_{z\in\mathcal Z_t}\mathbb E\left[\left(U(T,\xi\rho) - \xi\eta \rho z_T\right)1_{A} \right] = \sup\limits_{\xi\in\mathcal C_t(x)}\left(\sup\limits_{\phi\in\mathcal C_{tT}}\mathbb E\left[\left(U(T,\xi\phi\right)1_A \right]- \mathbb E\left[\xi\eta 1_{A} \right]\right),$$}
i.e., \eqref{9226} hold. This completes the proof of the lemma.
\end{proof}
\begin{Corollary}\label{cor9221}Let the condition of Lemma \ref{231} hold and $\eta \in\mathcal N_t$ be fixed. Then for every $A\in\mathcal F_t$,  we have
\begin{displaymath}
\inf\limits_{z\in\mathcal Z_t}\mathbb E\left[V(T,\eta z_T)1_{A}\right] = 
\sup\limits_{\xi\in\bigcup\limits_{x>0}\mathcal C_t(x)}\sup\limits_{\rho\in\mathcal C_{tT}}\mathbb E\left[\left(U(T,\xi\rho) - \xi\eta\right)1_{A} \right].
\end{displaymath}
\end{Corollary}
\begin{proof}
The assertion of the corollary follows from Lemma \ref{lem9221} and \ref{lem9222}.
\end{proof}
\begin{Lemma}\label{lemConjugacy}
Let the condition of Lemma \ref{231} hold and $\eta\in\mathcal N_t$ be fixed. Then, we have
\begin{equation}\label{9224}\esssup\limits_{\xi\in\bigcup\limits_{x>0}\mathcal C_t(x)}\left( u(\xi,t,T) - \xi\eta\right) = v(\eta, t,T).
\end{equation}
\end{Lemma}
\begin{proof}First, it follows from Lemma \ref{lemSubConjugacy} that for every $\xi\in\bigcup\limits_{x>0}\mathcal C_t(x)$, $\rho\in\mathcal C_{tT}$, and $z\in\mathcal Z_t$, we have
\begin{equation}\label{9222}
\mathbb E_t\left[U(T,\xi \rho) \right]- \xi\eta \leq \mathbb E_t\left[V(T,\eta z_T)\right].
\end{equation}
Let us fix $m\in\mathbb N$ and set
\begin{displaymath}
A_m:=\left\{ \esssup\limits_{\xi\in\bigcup\limits_{x>0}\mathcal C_t(x)}\esssup\limits_{\rho\in\mathcal C_{tT}}\left( \mathbb E_t\left[U(T,\xi \rho) \right]- \xi\eta\right) 
\leq
\essinf\limits_{z\in\mathcal Z_T}\mathbb E_t\left[V(T,\eta z_T)\right] - \frac{1}{m}\right\}\in\mathcal F_t.
\end{displaymath}
Then, in view of \eqref{9222}, for every $\xi\in\bigcup\limits_{x>0}\mathcal C_t(x)$, $\rho\in\mathcal C_{tT}$, and $z\in\mathcal Z_t$, we get
\begin{displaymath}
\mathbb E_t\left[U(T,\xi \rho) \right]- \xi\eta \leq \mathbb E_t\left[V(T,\eta z_T)\right] -\frac{1}{m} 1_{A_m}.
\end{displaymath}
Multiplying both sides by $1_{A_m}$, we obtain
\begin{displaymath}
\mathbb E_t\left[U(T,\xi \rho)1_{A_m} \right]- \xi\eta 1_{A_m}\leq \mathbb E_t\left[V(T,\eta z_T)1_{A_m}\right] -\frac{1}{m} 1_{A_m}.
\end{displaymath}
Taking the expectation, we deduce that
\begin{displaymath}
\mathbb E\left[\left(U(T,\xi \rho)- \xi\eta \right)1_{A_m}\right]\leq \mathbb E\left[V(T,\eta z_T)1_{A_m}\right] -\frac{1}{m} \mathbb P[A_m].
\end{displaymath}
As the above inequality holds for every $\xi\in\bigcup\limits_{x>0}\mathcal C_t(x)$, $\rho\in\mathcal C_{tT}$, and $z\in\mathcal Z_t$, 
 we get
\begin{displaymath}
\sup\limits_{\xi\in\bigcup\limits_{x>0}\mathcal C_t(x)}\sup\limits_{\rho\in\mathcal C_{tT}}\mathbb E\left[\left(U(T,\xi \rho)- \xi\eta \right)1_A\right]\leq \inf\limits_{z\in\mathcal Z_t}\mathbb E\left[V(T,\eta z_T)1_A\right] -\frac{1}{m} \mathbb P[A_m].
\end{displaymath}
Combining the latter inequality with the assertion of Corollary \ref{cor9221}, we obtain that $\mathbb P[A_m] = 0$. As $m\in\mathbb N$ is arbitrary, we conclude that
$$\left\{ \esssup\limits_{\xi\in\bigcup\limits_{x>0}\mathcal C_t(x)}\esssup\limits_{\rho\in\mathcal C_{tT}}\left( \mathbb E_t\left[U(T,\xi \rho) \right]- \xi\eta\right) 
<
\essinf\limits_{z\in\mathcal Z_T}\mathbb E_t\left[V(T,\eta z_T)\right] \right\} = \bigcup\limits_{m\in\mathbb N}A_m$$
has measure $0$. Equivalently, we have
\begin{equation}\nonumber
\esssup\limits_{\xi\in\bigcup\limits_{x>0}\mathcal C_t(x)}\left( \esssup\limits_{\rho\in\mathcal C_{tT}}\mathbb E_t\left[U(T,\xi \rho) \right]- \xi\eta\right) 
=
\essinf\limits_{z\in\mathcal Z_T}\mathbb E_t\left[V(T,\eta z_T)\right],\quad \Pas,
\end{equation}
and thus \eqref{9224} holds. This completes the proof of the lemma.
\end{proof}
\subsection{The reverse conjugacy}
The reverse conjugacy, or biconjugacy, between $u$ and $v$ is a subject closely related to the Fenchel-Moreau theorem. In the present context, this is a delicate topic, as $u$ and $-v$ are defined as essential suprema, and thus they take values in space of $\mathcal F_t$-measurable extended real-valued functions. Therefore, we cannot apply the standard biconjugacy results from convex analysis, e.g., of Rockafellar \cite{Rok}, directly. The topic of the Fenchel-Moreau theorem for $\bar{\mathbf L}^0$-valued functions has been studied recently, see \cite{Kupper17}. However,  the domains of $u$ and $v$ do not form a dual pair of Banach spaces, and thus these domains do not satisfy the assumptions of \cite{Kupper17}. Therefore, we have to prove biconjugacy by hand. 

\begin{Lemma}\label{lembiconjugacy}
Under the conditions of Lemma \ref{lemConjugacy}, for every \textcolor{black}{$\xi\in \bigcup\limits_{x>0}\mathcal C_t(x)$}, we have $$u(\xi, t,T) = \essinf\limits_{\eta \in \mathcal N_t}\left( v(\eta, t,T) + \xi\eta\right). $$
\end{Lemma}
\begin{proof}
The proof follows the proof of Lemma \ref{lemConjugacy} above with some minor modifications. Therefore, we do not present the complete proof and only highlight the differences. First, we need to pass from $Z_t$ to the closure of the convex and solid hull of $\{z_T:~z\in\mathcal Z_t\}$, which by the Fatou-convergence-type argument above, similarly to the proof of \cite[Proposition 3.1]{KS}, can be constructed as 
\begin{equation}\label{defDtT}\mathcal D_{tT} :=\left\{ h\in \tbl{m\mathcal F_T}:~h\leq z_T,~for~some~z\in\mathcal Z_t\right\}.
\end{equation} 
Then, (and this is the main step) we need to show that for a given $\xi\in \bigcup\limits_{x>0}\mathcal C_t(x)$ and $A\in\mathcal F_t$, we have
\begin{displaymath}
\lim\limits_{y\to\infty}\sup\limits_{\rho\in\mathcal C_{tT}}\inf\limits_{\eta z \in\mathcal D_T(y)}
\mathbb E\left[ \left(V(T,\eta z) + \eta z \xi\rho\right)1_A\right] \leq 
\sup\limits_{\rho\in\mathcal C_{tT}}\mathbb E\left[ U\left(T,\xi\rho\right)1_A\right].
\end{displaymath}
The latter can be obtained as follows. Let us consider strictly positive elements $\eta\in\mathcal D_t$ and $z\in\mathcal D_{tT}$ (where $\mathcal D_t$ and $\mathcal D_{tT}$ are defined in \eqref{defDt} and \eqref{defDtT}, respectively) such that 
\begin{equation}\label{9281}
\mathbb E\left[V(T,\eta z)\right]<\infty.
\end{equation}
The existence of such elements follows from Assumption \ref{finiteness} (combined with the argument in \textcolor{black}{\cite[Proposition 1]{KS2003}, where it can be proven that the infimum can be reached over the densities of the equivalent martingale measures under NFLVR, this argument combined with passing to the num\'eraire portfolio as a num\'eraire and stochastic utility, or equivalently, by treating the dual problem as in} the proof of \cite[Theorem 3.3]{Mostovyi2015}).

Further, by Assumption \ref{finiteness}, we there exist $\xi\in\mathcal C_t$ and $\rho\in\mathcal C_{tT}$, such that $\mathbb E\left[U(T,\xi\rho)\right]>-\infty$. Then for such $\xi$, $\rho$, $\eta$, and $z$, from the definition of $V$ and since  $0\leq\mathbb E\left[ \xi\rho\eta z\right]<\infty$ (which is a consequence of the respective definitions of $\mathcal D_t$ and $\mathcal D_{tT}$), we get
\begin{equation}\label{9291}
-\infty<\mathbb E\left[U(T,\xi\rho)\right] - y\mathbb E\left[ \xi\rho\eta z\right]\leq \mathbb E\left[ V(T,y\eta z)\right],\quad y>0,
\end{equation}
whereas from the monotonicity of $V(T,\cdot)$ and \eqref{9281}, we get
\begin{equation}\label{9292}
\mathbb E\left[ V(T,y\eta z)\right] \leq \mathbb E\left[ V(T,\eta z)\right]<\infty,\quad y\geq 1.
\end{equation}
Combining \eqref{9291} and \eqref{9292}, we deduce that (strictly positive elements) $\eta \in\mathcal N_t$ and $z \in\mathcal D_{tT}$ satisfy $$-\infty<\mathbb E\left[ V(T,y\eta z)\right] \leq \mathbb E\left[ V(T,\eta z)\right]<\infty,\quad y\geq 1.$$
Next, along the lines of the proof of Lemma \ref{996}, we define
\begin{displaymath}
U^n(T,x) :=\inf\limits_{0<y\leq n \eta z}\left( V(T,y) + xy\right),\quad (x,\omega)\in(0,\infty)\times \Omega,
\end{displaymath}
i.e., from $U$, we pass to a sequence of truncated stochastic fields (similar to the ones in the proof of \cite[Lemma 3.9]{Mostovyi2015}). With such $U^n$'s, as in the proof of Lemma \ref{996}, we can show that  
$$\lim\limits_{n\to\infty}\sup\limits_{\rho\in\mathcal C_{tT}}\mathbb E\left[ U^n(T,\xi\rho)1_{A}\right] = \sup\limits_{\rho\in\mathcal C_{tT}}\mathbb E\left[ U(T,\xi\rho)1_{A}\right].$$
Here the only challenge is to establish the uniform integrability of $(U^n)^{+}(T,\xi\rho)$, $\rho\in\mathcal C_{tT}$. This, however, follows from the following estimates (similar to the ones in \cite[Lemma 3.9]{Mostovyi2015}): for every $n\geq 2$, one can see that 
\begin{equation}\label{9283}
\begin{split}
U^n(T,x)&\leq U(T,x)1_{\{x>-V'(T,2\eta z)\}}\\
&\quad + 2(V(T,\eta z)-V(T,2\eta z))1_{\{x\leq -V'(T,2\eta z)\}},\quad (x,\omega)\in(0,\infty)\times \Omega.\\
\end{split}
\end{equation}
\eqref{9283} implies that for every $n\geq 2$, we have
$$U^n(T,x) \leq \max\left( U(T,x), 2V(T,\eta  z) - V(T,2\eta z)\right),\quad (x,\omega)\in(0,\infty)\times \Omega,$$
and the uniform integrability of $(U^n)^{+}(T,\xi\rho)$, $\rho\in\mathcal C_{tT}$, follows from \cite[Lemma 3.5]{Mostovyi2015}. The remaining parts of the proof of this lemma are very similar to the proof of Lemma \ref{lemConjugacy} and, therefore, they are skipped. 
\end{proof}


\begin{Lemma}\label{lem1291}
Let the condition of Lemma \ref{231} hold, and $\xi\in\tbl{\mathcal E^u_t}$. Then there exist $\hat\eta$ in is the closure of $\mathcal N_t$ in $\mathbb L^0$ and $\hat z\in\mathcal Z_t$, such that 
\begin{equation}\label{eqnConjugacy}
\begin{split}
u(\xi,t,T) &= \essinf\limits_{\eta \in \mathcal N_t}\left( v(\eta, t,T) + \xi\eta\right)\\
&=\mathbb E\left[ V\left(T,\hat \eta\hat z_T\right)|\mathcal F_t\right] + \hat \eta \xi.
\end{split}
\end{equation}
Further, let $\hat\rho$ be the optimizer to \eqref{primalProblem} corresponding to $\xi$.
 Then we have
\begin{equation}\label{1296}
\mathbb E_t\left[\hat \rho \hat z_T\right] = 1,\quad on\quad\{\xi>0\},\quad\mathbb P-a.s.
\end{equation}
and
\begin{equation}\label{1295}
U'(T,\xi\hat \rho) = \hat \eta \hat z_T\quad and \quad \xi\hat \rho = -V'(T,\hat \eta\hat z_T)\quad on\quad\{\xi>0\},\quad \mathbb P-a.s.
\end{equation}
\end{Lemma}
\begin{proof}
Let $\eta^1$, $\eta^2$ in $\mathcal N_t$ and $z^1$ and $z^2$ be the corresponding minimizers to \eqref{dualProblem}, then with
\begin{equation}\label{6131}
A :=\left\{
\mathbb E\left[V\left(T,\eta^1 z^1_T\right)|\mathcal F_t\right] + \eta^1 \xi <
\mathbb E\left[V\left(T,\eta^2 z^2_T\right)|\mathcal F_t\right] + \eta^2 \xi\right\}\in\mathcal F_t,
\end{equation}
and via Lemmas \ref{lemNconvex} and \ref{lemG}, we get
\begin{displaymath}\begin{split}
\eta &:=\eta^1 1_{A} + \eta^2 1_{A^c}\in\mathcal N_t,\quad as~well~as\\
 z &:=\frac{1_A \eta^1z^1 + 1_{A^c}\eta^2 z^2}{\eta}{1_{\{\eta\neq 0\}}+\left(1_A  z^1 + 1_{A^c}z^2\right)1_{\{\eta = 0\}}}\in\mathcal Z_t.
\end{split}
\end{displaymath} 
In turn, by concavity of $V$, and using Lemma \ref{lemG}, we obtain
\begin{displaymath}
\begin{split}
\mathbb E\left[V\left(T,\eta z_T \right)|\mathcal F_t\right] + \eta \xi&=\left( \mathbb E\left[V\left(T,\eta z_T \right)|\mathcal F_t\right] +  \eta \xi\right)1_A +\left(\mathbb E\left[V\left(T,\eta z_T \right)|\mathcal F_t\right] + \eta \xi\right)1_A^c\\
&=\left(\mathbb E\left[V\left(T,\eta z_T1_A  \right)|\mathcal F_t\right] + \eta^1 \xi\right)1_A +\left(\mathbb E\left[V\left(T,\eta z_T1_A^c \right)|\mathcal F_t\right]+\eta^2\xi \right)1_A^c\\
&= \left(\mathbb E\left[V\left(T,\eta^1 z^1_T  \right)|\mathcal F_t\right] + \eta^1 \xi\right)1_A +\left(\mathbb E\left[V\left(T,\eta^2 z^2_T \right)|\mathcal F_t\right]+\eta^2\xi \right)1_A^c\\
&=\min\left(\mathbb E\left[V\left(T,\eta^1 z^1_T  \right)|\mathcal F_t\right] + \eta^1 \xi, \mathbb E\left[V\left(T,\eta^2 z^2_T \right)|\mathcal F_t\right]+\eta^2\xi \right),
\end{split}
\end{displaymath}
where in the last equality, we have used \eqref{6131}. Therefore, using \cite[Theorem A.2.3, p. 215]{Pham09}, we deduce the existence of a sequence 
$\eta^nz^n$, $n\in\mathbb N$, such that
\begin{equation}\label{6132}
\lim\limits_{n\to\infty}\left(\mathbb E\left[V\left(T,\eta^n z^n_T \right)|\mathcal F_t\right] + \eta^n \xi\right)=\essinf\limits_{\eta z\in\mathcal G_t}\left(\mathbb E\left[V\left(T,\eta z_T \right)|\mathcal F_t\right] + \eta \xi\right),\quad \mathbb P-a.s.,
\end{equation}
where we recall that $\mathcal G_t$ is defined in \eqref{defGt}.
By passing to convex combinations, and by applying Lemma \ref{lemG}, which asserts that a convex combination of elements of $\mathcal G_t$ is an element of $\mathcal G_t$, we may obtain a subsequence of elements of $\mathcal G_t$, which we still denote $\eta^n z^n_T$, $n\in\mathbb N$, and which converges a.s. to a limit, which we denote by $\psi$.

Let us consider $\eta^n$, $n\in\mathbb N$. By passing to convex combinations, we may obtain  a family of convex weights, $\lambda^n_k\in[0,1]$, $k \in\{ 0,\dots, M_n\}$, $n\in\mathbb N$, where $M_n\in\mathbb N$, such that for every $n\in\mathbb N$, $\sum\limits_{k = 0}^{M_n} \lambda^n_k = 1$, and $\tilde \eta^n \triangleq\sum\limits_{k = 0}^{M_n}\lambda^n_k\eta^{n+k}$, $n\in\mathbb N$, converges along a subsequence to a limit, $\hat \eta$, a.s. Applying the same convex weights to $\eta^nz^n$, 
and passing to the same subsequence\footnote{$\lambda$ in \eqref{8231} and \eqref{312} (of Lemma \ref{lemG}) is the same.}, via Lemma \ref{lemG}, we get 
$$\sum\limits_{k = 0}^{M_n}\lambda^n_k\eta^{n+k}z^{n+k}_T = \tilde \eta^n\tilde z^n_T,\quad n\in\mathbb N,$$
for some $\tilde \eta^n\in\mathcal N_t$ and $\tilde z^n\in\mathcal Z_t$, $n\in\mathbb N$. As $\eta^n z^n_T$, $n\in\mathbb N$, converges a.s. to $\psi$, $\tilde \eta^n\tilde z^n_T$, $n\in\mathbb N$, also converges a.s. to the same limit $\psi$. As both $\tilde \eta^n$ and $\tilde \eta^n\tilde z^n_T$, $n\geq 1$, converge, we additionally obtain that $\tilde z^n_T$, $n\in\mathbb N$, converges to a limit, which we denote $\hat z_T$. Therefore, 
we have
\begin{equation}\label{2283}
\psi = \lim\limits_{n\to\infty}\tilde \eta^n\tilde z^n_T = \hat \eta\hat z_T.
\end{equation}
As in the proof of Lemma \ref{lemDualExistence}, one may show that $\hat z_T$ is the terminal value of the element of $\mathcal Z_t$, and thus $\hat z\in\mathcal Z_t$. 

Now, for an arbitrary $\tilde A\in\mathcal F_t$, convexity of $V$, Lemma \ref{lembiconjugacy}  and \eqref{6132} imply that
\begin{displaymath}
\begin{split}
&\lim\limits_{n\to\infty}\mathbb E\left[\left( V(T,\tilde \eta^n\tilde z^n) + \tilde \eta^n\tilde z^n \xi\hat \rho - U(T,\xi\hat \rho)\right)1_{\tilde A}\right] \\
&\leq\lim\limits_{n\to\infty}\mathbb E\left[\left( V(T,\tilde \eta^n\tilde z^n) + \tilde \eta^n \xi - U(T,\xi\hat \rho)\right)1_{\tilde A}\right] = 0.
\end{split}
\end{displaymath}
From the definition of $V$, for every $n\geq 1,$ we have
$$ 0\leq V(T,\tilde \eta^n\tilde z^n_T) + \tilde \eta^n\tilde z^n_T \xi\hat \rho - U(T,\xi\hat \rho),\quad \mathbb P-a.s.,$$
Therefore, Fatou's lemma and \eqref{2283} give
\begin{equation}\label{2284}\mathbb E\left[\left( V(T,\hat \eta\hat z_T) + \hat \eta\hat z_T \xi\hat \rho - U(T,\xi\hat \rho)\right)1_{\tilde A}\right] = 0,\quad \tilde A\in\mathcal F_t.
\end{equation}
\eqref{eqnConjugacy} follows.
In turn, \eqref{2284} together with 
\begin{equation}\nonumber
V(T,\hat \eta\hat z_T) + \hat \eta\hat z_T \xi\hat \rho  - U(T,\xi\hat \rho)\geq 0,\quad \mathbb P-a.s.
\end{equation}
implies that 
$$V(T,\hat \eta\hat z_T) + \hat \eta\hat z_T \xi\hat \rho  =U(T,\xi\hat \rho),\quad \mathbb P-a.s.,$$
which,  via the definition of $V$, implies \eqref{1295}. In turn, \eqref{1296} follows from the polar structure of $\mathcal C_{tT}$ and $\mathcal Z_t$. 
%
\end{proof}
\begin{Lemma}\label{lem4191}Let the condition of Lemma \ref{231} hold, $\xi\in\tbl{\mathcal E^u_t}$. Then, with $\hat \rho$, $\hat \eta$, and $\hat z_T$ being as in Lemma \ref{lem1291}, we have: both $U(T,\xi\hat\rho)$ and $V^{+}\left(T,\hat \eta\hat z_T\right)$ belong to $\mathbb L^1(\mathbb P)$. Further, there exists $A_m$, $m\in\mathbb N$, an $\mathcal F_t$-measurable partition of $\Omega$, such that $\hat\eta 1_{A_m}\in \mathcal N_t$  and $V(T,\hat \eta\hat z_T)1_{A_m}\in\mathbb L^1(\mathbb P),$ $m\in\mathbb N$.  
\end{Lemma}
\begin{proof}
First, by \cite{Mostovyi2015}, $U^{+}(T,\xi\hat \rho)\in\mathbb L^1(\mathbb P)$. By assumption \eqref{effDomP}, there exists $\rho\in\mathcal C_{tT}$, such that $\mathbb E\left[U^{-}\left(T,\xi\rho\right)\right]<\infty$. From optimality of $\hat \rho$, we get
\begin{equation}\label{3131}
\mathbb E_t\left[ U^{-}\left(T,\xi\hat \rho\right)\right]\leq \mathbb E_t\left[ U^{+}\left(T,\xi\hat \rho\right)\right] - \mathbb E_t\left[ U^{+}\left(T,\xi \rho\right)\right] + \mathbb E_t\left[ U^{-}\left(T,\xi \rho\right)\right],
\end{equation}
As $\mathbb E\left[ U^{+}\left(T,\xi \rho\right)\right]<\infty$ (by \cite{Mostovyi2015}) and $\mathbb E\left[ U^{-}\left(T,\xi \rho\right)\right]<\infty$ (by \eqref{effDomP}), we deduce from \eqref{3131} that 
$U^{-}\left(T,\xi\hat \rho\right)\in\mathbb L^1(\mathbb P)$. Therefore, $U(T,\xi\hat\rho)\in \mathbb L^1(\mathbb P).$

To show that $V(T,\hat\eta \hat z_T)\in\mathbb L^1(\mathbb P)$, we use the equality
\begin{equation}\label{3132}
V(T,\hat \eta\hat z_T) + \xi\hat \rho\hat \eta\hat z_T = U(T,\xi\hat \rho),\quad \mathbb P-a.s.
\end{equation}
By taking positive part in \eqref{3132}, we obtain
\begin{displaymath}
V^{+}\left(T,\hat \eta\hat z_T\right) \leq \left(V\left(T,\hat \eta\hat z_T\right) + \xi\hat \rho\hat \eta\hat z_T\right)^{+} = U^{+}\left( T,\xi\hat \rho\right)\in\mathbb L^1(\mathbb P).
\end{displaymath}
Likewise
\begin{displaymath}
\left(V\left(T,\hat \eta\hat z_T \right)+ \xi\hat \rho\hat \eta\hat z_T \right)^{-}= U^{-}\left(T,\xi\hat \rho\right)\in\mathbb L^1(\mathbb P).
\end{displaymath}
To show the existence of an $\mathcal F_t$-measurable partition of $\Omega$, such that $\hat\eta 1_{A_m}\in\mathcal N_t$, $m\geq 1$, first, we observe that it follows from Assumption \ref{NUPBR}, there exists a strictly positive element $\eta_0\in\mathcal N_t$. Let $(\eta_m)_{m\in\mathbb N}\subset\mathcal N_t$ be a sequence, which converges to $\hat \eta$, $\mathbb P$-a.s.
Let $$A_0 :=\emptyset \quad and \quad
A_m :=\left\{ \hat \eta\leq 2\eta_m + \eta_0\right\}\backslash\bigcup\limits_{k = 0}^{n-1}A_k,\quad m\geq 1.$$
Then, by construction on each $A_m$, $\hat \eta1_{A_m}\in\mathcal N_t$, and $A_m$'s are disjoint subsets $\mathcal F_t$. Finally, from $\mathbb P$-a.s. convergence of $\eta_m$, $m\geq 1$, to $\hat \eta$, we deduce that $\mathbb P\left[\left(\bigcup\limits_{m\geq 1}A_m\right)^c \right]= 0$.
Now, $V(T,\hat \eta\hat z_T)1_{A_m}\in\mathbb L^1(\mathbb P)$ by the integrability of $U(\xi\hat \rho)$ and \eqref{3132}. 
\end{proof}
\tbl{
\subsection{Some results from \cite{Mostovyi2015} used above}\label{resMos15}
Some results from \cite{Mostovyi2015} used above are given below in an adjusted form. Let $\mathbb P$ be a probability measure on a measurable space  $(\Omega,
\mathcal F)$. Denote by
$\mathbf{L}^0=\mathbf{L}^0\left(
  \Omega, \mathcal F, \mathbb P\right)$ the vector space of 
random variables on 
$(\Omega, \mathcal F, \mathbb P)$ endowed with the topology of convergence in probability measure $\mathbb P$.
Let $\mathbf{L}^0_{+}$ denote
 the set of nonnegative random variables on $(\Omega, \mathcal F, \mathbb P)$. 
Let $\mathcal{C}, \mathcal{D}$ be polar subsets of
$\mathbf{L}^0_{+}$, that is 
}

\tbl{
  \begin{equation}\label{polarityCD}
    \begin{split}
      \xi\in\mathcal{C} &{\quad\rm if~and~only~if}\quad \mathbb E\left[ \xi \eta \right] \leq 1\quad{\rm
        for~every~}\eta \in\mathcal{D},\\
      \eta\in\mathcal{D} &{\quad \rm if~and~only~if}\quad\mathbb E\left[ \xi \eta \right] \leq 1
      \quad{\rm
        for~every~}\xi \in\mathcal{C}.
    \end{split}
  \end{equation}
and we additionally suppose that 
   \begin{equation}\label{positivityCD}
     \begin{array}{rcl}
       {\rm there~exists~}&\xi\in\mathcal{C}& {\rm~such~that~}\xi>0,\\
       {\rm there~exists~}&\eta\in\mathcal{D}& {\rm~such~that~}\eta>0.\\
     \end{array}
   \end{equation}
}

\tbl{
Let us notice a {symmetry} between the sets $\mathcal C$ and $\mathcal D$. 
For $x>0$ and $y>0$ one can define the sets:
\begin{equation}\label{defCxDy}
  \begin{split}
    \mathcal{C}(x) &:= x\mathcal{C} :=\left\{ x\xi: ~\xi\in\mathcal{C}\right\}, \\
    \mathcal{D}(y) &:= y\mathcal{D} :=\left\{ y\eta: ~\eta\in\mathcal{D}\right\}. 
  \end{split}
\end{equation}
}

\tbl{
Let us consider a \textit{stochastic utility function}
$U$: $\Omega\times[0, \infty)\to \mathbb{R}\cup\{-\infty\}$ satisfying  Assumption \ref{Assumption2} below. 
\begin{Assumption}\label{Assumption2}
  For every $ \omega\in \Omega$, the function $x\to U( \omega, x)$ is strictly increasing, 
  strictly concave,  differentiable on $(0,\infty)$, and
  satisfies the Inada conditions:
  \begin{equation}\label{InadaAbstract}
    \lim\limits_{x\searrow 0}U'(\omega, x) =
    \infty \quad{\rm and}\quad \lim\limits_{x\to
      \infty}U'(\omega, x) = 0,
  \end{equation}
where $U'(\cdot, \cdot)$ denotes the partial derivative with respect
  to the second argument. At $x=0$ we suppose that $U(\omega, 0) = \lim\limits_{x\searrow
    0}U(\omega, x)$, this value may be $-\infty$. For every $x\geq0$, we suppose that $U\left(
    \cdot, x \right)$ is measurable.
\end{Assumption}
Define the \textit{conjugate function} $V$ to $U$ as
\begin{equation}\nonumber
  V(\omega, y) := \sup\limits_{x>0}\left( U(\omega, x) - xy \right)
  ,\quad \left(\omega, y \right) \in \Omega\times[0,
  \infty).
\end{equation}
Observe that $-V$ satisfies Assumption \ref{Assumption2}.
}

\tbl{
Now we can state the optimization problems:
\begin{equation}\label{primalProblem2}
  u(x) = \sup\limits_{\xi\in\mathcal{C}(x)}\mathbb E\left[ {U}(\xi)\right], \quad x>0,
\end{equation}
\begin{equation}\label{dualProblem2}
  v(y) = \inf\limits_{\eta\in\mathcal{D}(y)}\mathbb E\left[ {V}(\eta)\right], \quad y>0.
\end{equation}
}




\tbl{
\begin{customthm}{3.2 in \cite{Mostovyi2015}}\label{mainTheorem2}
  Assume that $\mathcal{C}$ and $\mathcal{D}$ satisfy conditions
  (\ref{polarityCD}) and (\ref{positivityCD}). Let Assumption
  \ref{Assumption2} hold and suppose
  \begin{equation}\label{mainCondition2}
    v(y)<\infty~~for~ all~y>0\quad and\quad u(x) > -\infty~~for~ all~x>0.
  \end{equation}
  Then we have:
  \begin{enumerate}
  \item $u(x)< \infty$ for all $x>0,$ $v(y)>-\infty$ for all $y>0.$
    The functions $u$ and $v$ satisfy the biconjugacy relations, i.e.,
    \begin{equation}\label{biconjugacy2}
      \begin{array}{rcl}
        v(y) &=& \sup\limits_{x>0}\left(u(x) - xy\right),\quad y>0,\\
        u(x) &=& \inf\limits_{y>0}\left(v(y) + xy\right),\quad x>0.\\
      \end{array}
    \end{equation}
    The functions $u$ and $-v$ are strictly increasing, strictly concave, continuously differentiable on $(0,
    \infty)$, and satisfy the Inada
    conditions:
    \begin{equation}\nonumber
      \begin{split}
        \lim\limits_{x\searrow 0}u'(x) = \infty, & \quad \lim\limits_{y\searrow 0}-v'(y) = \infty,\\
        \lim\limits_{x\to\infty}u'(x) = 0, & \quad         \lim\limits_{y\to\infty}-v'(y) = 0.
      \end{split}
    \end{equation}
  \item For every $x>0$ the optimal solution $\hat{\xi}(x)$ to (\ref{primalProblem2})
    exists  and is
    unique. For every $y>0$ the optimal solution $\hat{\eta}(y)$ to
    (\ref{dualProblem2})  exists and is unique. If $y=u'(x)$, we have the dual
    relations
    \begin{displaymath}
        \hat{\eta}(y) = {U}'\left(\hat{\xi}(x) \right),
        \quad\mathbb P-{\rm~a.s.}
    \end{displaymath}
and
    \begin{displaymath}
      \mathbb E\left[\hat{\xi}(x)\hat{\eta}(y)\right] = xy.
      \end{displaymath}
 \end{enumerate}
\end{customthm}
}
\tbl{
Let
$\tilde{\mathcal{D}}$ be a subset of $\mathcal{D}$
such that
\begin{itemize}
\item [(i)]
$\tilde{\mathcal{D}}$ is closed with respect to countable convex
combinations,
\item [(ii)]
We have
\end{itemize}
\begin{equation}\nonumber
  \sup\limits_{\eta\in\mathcal{D} }\mathbb E\left[\xi\eta \right] =
  \sup\limits_{\eta\in\tilde{\mathcal{D}} }\mathbb E\left[ \xi\eta \right], \quad \xi\in\mathcal{C}.
\end{equation}
}

\tbl{
The statement of \cite[Theorem 3.3]{Mostovyi2015} (the part of \cite[Theorem 3.3]{Mostovyi2015} that was used above). 
\begin{customthm}{3.3 in \cite{Mostovyi2015}}\label{secondTheorem2}
Under the conditions of Theorem \ref{mainTheorem2}, we have
    \begin{displaymath}
      \begin{split}
      v(y) & = \inf\limits_{\eta\in\mathcal{\tilde{D}} }
      \mathbb E\left[
      {V}\left(y\eta\right)\right], \quad y>0,\\
     \end{split}
    \end{displaymath}
\end{customthm}
}
\tbl{
\begin{customlemma}{3.5 in \cite{Mostovyi2015}}\label{uiOfVminus}
  Under the conditions of Theorem \ref{mainTheorem2},
  for every $y>0$ the family $\left( {V}^{-}\left(h\right)\right)_{h\in\mathcal{D}(y)}$
  is uniformly integrable.
\end{customlemma}
\begin{customlemma}{3.9 in \cite{Mostovyi2015}}\label{conjugacy}Under the assumptions of Theorem
  \ref{mainTheorem2}, we have
\begin{equation}\label{58121}
v(y)=\sup\limits_{x>0}\left(u(x) - xy\right),\quad y>0.
\end{equation}
\end{customlemma}
}

\section{Stability analysis}\label{secStab}
Here we assume that the $0$-market consists of one \tbl{risky and one riskless asset (whose price still equals to $1$ at all times).} 
Let $M$, a one-dimensional continuous local martingale, that drives the return process of the \tbl{risky asset}. Throughout this section, $T>0$ is fixed. Then the dynamic of the \tbl{risky asset} is given by
\begin{equation}\label{3301}
R^{0} = M + \lambda\cdot \langle M\rangle
\end{equation}
\tbl{where $\lambda$ is a predictable process such that \begin{equation}\label{2161}
\lambda^2\cdot  \langle M\rangle_T<\infty,\quad \mathbb P-a.s.\end{equation}} Thus, the return of the \tbl{risky asset} (from section \ref{secModel}) has the form \eqref{3301}. For the absence of arbitrage in the sense of Assumption \ref{NUPBR}, finite variation part of the return process has to be absolutely continuous with respect to the  quadratic variation of its martingale part, see \cite{HulleySchweizer}.
We suppose that the riskless asset stays unperturbed and consider perturbed family of returns of \tbl{risky asset}s of the form
\begin{equation}\label{perR}
R^\varepsilon = (1 + \varepsilon\psi)\cdot(M + \lambda(1 + \varepsilon\theta )\cdot \langle M\rangle), \quad \varepsilon\in\mathbb R,
\end{equation}
\tbl{where $\psi$ and $\theta$ are some predictable processes, such that \begin{equation}\label{2162}
\theta^2\cdot \langle M\rangle<\infty,\quad \mathbb P-a.s.,
\end{equation} and $|\psi|$ is uniformly bounded.}
\textcolor{black}{
Perturbations of the input model parameters might appear due to errors in the estimation of the model parameters under a statistical procedure. In connection to many models of the stock price used in practice, $\psi$ corresponds to perturbations of the volatility, and then, ones $\phi$ is fixed, $\theta$ governs the distortions of the drift of the \tbl{risky asset}. We discuss a connection to a different parametrization of perturbations in the following remark. Mathematically, the closest paper, where such perturbations occur, are \cite{MostovyiNumeraire} and \cite{MostovyiSirbuModel} (the case of $\psi\equiv 0$).
}

\begin{Remark}\label{remPert1}
Parametrization of perturbations in the form \eqref{perR} is closely related to the ones that appear in the literature more often:
\begin{equation}\label{perR2}
\begin{split}
R^{\varepsilon} &= R^{0} + (\varepsilon\psi)\cdot M + (\varepsilon \nu)\cdot \langle M\rangle\\
&= (1 + (\varepsilon\psi))\cdot M + (\lambda+ (\varepsilon \nu))\cdot \langle M\rangle.\\
\end{split}
\end{equation}
Here $\varepsilon\psi$ amount to perturbations of the martingale part (of volatility in the simplest settings) and $\varepsilon \nu$ to perturbations of the finite variation part (or drift) of the return of the \tbl{risky asset}. 
From \eqref{perR2}, one can arrive to \eqref{perR} by assuming that that $\nu$ (linearly) depends on $\varepsilon$, and by making the following reparametrization:
\begin{equation}\label{reparametrization}
\nu = \lambda\nu',\quad \nu' = \psi + \nu'', \quad \nu'' = \theta(1 + \varepsilon\psi).
\end{equation}
The reason for imposing \eqref{perR} instead of \eqref{perR2} is a simple structure of integrability Assumption \ref{asPert} and no issues related to differentiation with respect to a parameter under stochastic integration, as in \cite{Metivier82} and \cite{HuttonNelson}. 
We give further details on how to get the stability results with \eqref{perR2} in Remarks \ref{remPert2} and \ref{remPert3} below.  
\end{Remark}
Now we fix a proportion of the total wealth invested in the \tbl{risky asset}\footnote{Fixing the proportion of wealth invested in the \tbl{risky asset} leads to the admissibility of the corresponding wealth process for every $\eps\neq 0$. If one fixes the number of shares of the \tbl{risky asset} in the portfolio, then for $\eps\neq 0$, under both parametrizations \eqref{perR} and \eqref{perR2}, the associated wealth process can be negative with positive probability, in general.} and investigate the dynamic behavior of the indirect utility under small perturbations of the drift and volatility of the underlying \tbl{risky asset} as in \eqref{perR}.
To make this mathematically precise
, we extend the definitions from section \ref{secModel} in a natural way as follows:
$$\mathcal X^\varepsilon(x) \triangleq\left\{X\geq 0:~X = x + H\cdot R^\eps~for~some~R^\eps-integrable~H\right\},\quad x\geq 0,~\eps\in\mathbb R.$$
For every $\eps\in\mathbb R$, the initial wealth $\bar x\geq 0$ and a {\textcolor{black}{predictable and locally bounded} process $\pi$ are the same, but the corresponding family of the wealth processes alters due to different integrators $R^\eps$, i.e., we consider the family
 $$X^{\pi, \eps} = \bar x\mathcal E\left( \pi\cdot R^{\eps}\right),\quad \eps\in\mathbb R.$$
Likewise, for every $\eps\in\mathbb R$, the set of wealth processes in $\mathcal X^\eps(\bar x)$, which equal to  $X^{\pi,\eps}$ on $[0,t]$, is denoted by $\mathcal A^\eps(X^{\pi,\eps}_t, t)$, that is
\begin{equation}\label{defAeps}
\mathcal A^\eps(X^{\pi,\eps}_t, t) :=\left\{\tilde X\in\mathcal X^\eps(\bar x):~ \tilde X_s = X^{\pi,\eps}_s,~for~s\in[0,t],~\mathbb P-a.s.
\right\},\quad \eps\in\mathbb R.
\end{equation}
Finally the family of dynamic indirect utilities associated with $\pi$ up to $T$ is defined as
\begin{equation}\nonumber
u^\eps(X^{\pi,\eps}_t, t, T) :=\esssup\limits_{\tilde X\in\mathcal A^\eps(X^{\pi,\eps}_t, t)}\mathbb E\left[ U\left( T,\tilde X_T\right)|\mathcal F_t\right], \quad t\in[0,T],~\eps\in\mathbb R,
\end{equation}
for brevity, we denote
\begin{equation}\label{3305}
J^{\eps, T}_t :=u^\eps(X^{\pi,\eps}_t, t, T), \quad t\in[0,T],~\eps\in\mathbb R.
\end{equation}
For every $\varepsilon \in\mathbb R$, let us set 
\begin{equation}\label{defEta}
\eta^\varepsilon = -\varepsilon\lambda\theta
\quad and\quad 
L^{\varepsilon} :=\mathcal E\left( \eta^{\varepsilon}\cdot R^{0} \right).
\end{equation}
Note that $L^{\varepsilon}\in\mathcal X^0(1)$ for every $\varepsilon$. The processes $L^\eps$ drive the correction terms in Proposition \ref{thmDeriv}.

\begin{Remark}\label{remPert2}
If one chooses perturbations \eqref{perR2}, 
the $\eta^\varepsilon$ and $L^\varepsilon$ should be defined as
\begin{equation}\nonumber
{\eta}^{\varepsilon} :=\left(\lambda - \frac{\lambda + \varepsilon \nu}{1 + \varepsilon \psi}\right)\quad and \quad L^{\varepsilon} :=\mathcal E\left(\eta^{\varepsilon}\cdot R^{0} \right).
\end{equation}
This leads to the same heuristic limiting formulas, but stronger integrability conditions (than the one in Assumption \ref{asPert}) are needed to complete proofs. 
\end{Remark}
\subsection{Heuristic derivative of $L^\varepsilon$} 
We denote 
\begin{equation}\label{defBarR}
\bar R = (\lambda \theta)\cdot R^0.
\end{equation}
Then $$L^\varepsilon = \mathcal E\left( \eta^\varepsilon\cdot R^0\right) =  \mathcal E\left( - \varepsilon\cdot \bar R\right) = \exp\left(-\varepsilon \bar R - \tfrac{1}{2}\varepsilon^2 \langle \bar R\rangle\right).$$
We set $F:=- \bar R_T$.
Therefore, we get
\begin{displaymath}
\frac{\partial L^\varepsilon_T}{\partial \varepsilon} = L^\varepsilon_T\left(-\bar R_T - \varepsilon \langle\bar R\rangle_T\right),\quad \left.\frac{\partial L^\varepsilon_T}{\partial \varepsilon} \right|_{\varepsilon = 0} = -\bar R_T = F.
\end{displaymath}
Similarly, we obtain
\begin{displaymath}
\frac{\partial\frac{1}{ L^\varepsilon_T}}{\partial \varepsilon} = \frac{1}{L^\varepsilon_T}\left(\bar R_T + \varepsilon \langle\bar R\rangle_T\right),\quad \left.\frac{\partial\frac{1}{ L^\varepsilon_T}}{\partial \varepsilon} \right|_{\varepsilon = 0} = \bar R_T = -F.
\end{displaymath}

\begin{Remark}\label{remPert3}
Here we show that under \eqref{perR2}, we get the same heuristic formulas for the derivatives of $L^\varepsilon$. 
We recall that under \eqref{perR2}, the corresponding $\eta^\varepsilon$ and $L^\varepsilon$ are given by 
$$\eta^{\varepsilon} :=\left(\lambda - \frac{\lambda + \varepsilon \nu}{1 + \varepsilon \psi}\right)\quad and \quad L^{\varepsilon} :=\mathcal E\left(\eta^{\varepsilon}\cdot R^{0} \right).
$$
Then, by direct computations and via the reparametrization formulas \eqref{reparametrization}, we obtain
\begin{displaymath}
\left.\frac{\partial L^\varepsilon_T}{\partial \varepsilon}\right|_{\varepsilon = 0} 
 = F\quad and\quad 
\left.\frac{\partial \frac{1}{L^\varepsilon_T}}{\partial \varepsilon}\right|_{\varepsilon = 0} = - F.
\end{displaymath}
Thus, the asymptotic behavior of $L^{\eps}$'s is similar under both parameterizations of perturbations. 
\end{Remark}
\subsection{Rigorous derivation}
For $t\in[0,T]$, and $\hat\rho$ and $\hat z_T$ being associated with $X^{\pi, 0}_t$ (via Lemma \ref{lem1291}), 
let us define the probability measure $\mathbb R^t$ as 
\begin{equation}\label{defMR}
\left.\frac{d \mathbb R^t}{d\mathbb P}\right|_{\mathcal F_s} := 1_{\{s\in[0,t]\}} + \mathbb E_s\left[\hat \rho\hat z_T\right]1_{\{s>t\}},\quad s\in[0,T].
\end{equation}
Note that Lemma \ref{lem1291} (via \eqref{1296}) ensures that $\mathbb R^t$ is a probability measure. 
Finally, we suppose that the perturbations are sufficiently bounded in the following sense.\begin{Assumption}\label{asPert}
The process $|\psi|$ is uniformly bounded from above and there exists a constant $\bar c>0$, such that 
$$\exp\left(\bar c( |\bar R_T| + \langle \bar R\rangle_T)\right)\in\mathbb L^1(\mathbb R^t),\quad t\in[0,T],$$
where the probability measure $\mathbb R^t$ and the process $\bar R$ are defined in \eqref{defMR} and  \eqref{defBarR}, respectively.
\end{Assumption}
We also need to strengthen the assumptions on $U$.
\begin{Assumption}\label{rra}
For every $\omega\in \Omega$, $U(T, \cdot)$ is  a strictly concave, strictly increasing, continuously differentiable, and 
there exist positive constants $\gamma_1>0$ and  $\gamma_2>0$, such that for every $x>0$ and $z\in(0,1]$, we have
\begin{equation}\label{2191}
\begin{split}
U'(T, zx)\leq z^{-\gamma_1} U'(T, x)
\quad  {\rm and} \quad 
-V'(T, zx)\leq -V'(T, x)z^{-\gamma_2}.
\end{split}
\end{equation}
For every $x\geq 0$, $U(T,x)$ is measurable. 
\end{Assumption}
\begin{Remark}
Assumption \ref{rra} holds if either relative risk aversion, \tbl{ $A(x):=-\frac{U''(T, x)x}{U'(T, x)}$, $x>0$,} or relative risk tolerance  \tbl{of $U(T,\cdot)$} \tbl{at $x$, given by $-\frac{yV''(T,y)}{V'(T,y)}$ for $y = U'(T,x)$, $x>0$,}  is bounded away from $0$ and $\infty$ uniformly in $\omega\in\Omega$, see e.g., \cite[Lemma 5.12]{MostovyiSirbuModel}.
\end{Remark}\tbl{
\begin{Remark}\label{remInada}
Condition \eqref{2191} implies the Inada conditions. This can be shown as follows. Let us fix $\omega\in\Omega$. Applying $U'(T,\cdot)$ to both sides of the second inequality in \eqref{2191}, and since $U'(T,\cdot)$ is decreasing, we get:
\begin{equation}\label{2205}
zx \geq U'(T, z^{-\gamma_2}(-V'(T,x))),\quad x>0, z\in(0,1].
\end{equation}
Now for $x = U'(T,1)$, $-V'(T,x)=1$, and, in \eqref{2205}, we have
$$zU'(T,1) \geq U'(T, z^{-\gamma_2}),\quad z\in(0,1].$$
Taking the limit as $z\searrow 0$, we deduce that 
$$0 \geq \lim\limits_{\tilde x\to\infty}U'(T, \tilde x).$$
Similarly, from the first inequality in \eqref{2191}, applying $-V'(T,\cdot)$ to both sides, and since $-V'(T,\cdot)$ is decreasing, we get
\begin{equation}\label{2206}
zx \geq -V'(T,z^{-\gamma_1}U'(T, x)),\quad x>0,z\in(0,1].
\end{equation}
For $x = -V'(T, 1)$, we have $U'(T, x) = 1$, and therefore in \eqref{2206}, we obtain
$$-V'(T, 1)z \geq -V'(T,z^{-\gamma_1}),\quad z\in(0,1].$$
Taking the limit as $z\searrow 0$, we deduce that 
$$0\geq -\lim\limits_{\tilde z\to\infty}V'(T, \tilde z).$$
By conjugacy between $U(T,\cdot)$ and $V(T,\cdot)$, the latter inequality implies that $\lim\limits_{\tilde x\to 0}U'(T,\tilde x) = \infty$.
\end{Remark}
}

\tbl{
\begin{Lemma}\label{lemMichael}
Let $T>0$ be fixed and consider a family of \tbl{risky asset}s parametrized by $\varepsilon\in\mathbb R$, whose returns are given by  \eqref{perR}. Let us suppose the validity of Assumption \ref{asPert} and \ref{rra}, and  
\begin{equation}\label{2192}
u^0(z,0,T)>-\infty\quad {\rm and} \quad \sup\limits_{x>0}\left( u^0(x,0,T) - xz\right)<\infty,\quad z>0.
\end{equation}
 Then,  
there exists $\varepsilon_0>0$, such that for every $\varepsilon \in(-\varepsilon_0, \varepsilon_0)$,  
the pair of traded assets, whose returns are given by $0$ and $R^\varepsilon$, satisfy NUPBR, and 
 \begin{equation}\label{2163}
 u^\varepsilon(z,0,T)>-\infty\quad {\rm and} \quad \sup\limits_{x>0}\left( u^\varepsilon(x,0,T) - xz\right)<\infty,\quad z>0,
 \end{equation}
 that is both Assumptions \ref{NUPBR} and \ref{finiteness} hold for every $\varepsilon \in(-\varepsilon_0, \varepsilon_0)$.
\end{Lemma}
\begin{Remark}
In particular, in view or Remark \ref{remInada}, under the conditions of 
Lemma \ref{lemMichael}, the results  of Section \ref{secDualChar} apply to perturbed models, for every  $\varepsilon$ in some neighborhood of the origin.  
\end{Remark}}

\begin{proof}[Proof of Lemma \ref{lemMichael}]\tbl{
Conditions \eqref{2161} and \eqref{2162} imply that no unbounded profit with bounded risk holds for both  the unperturbed model (corresponding to $\varepsilon = 0$) and perturbed models ($\varepsilon \neq 0$), as $\mathcal E\left( -(\lambda(1 + \varepsilon \theta))\cdot M\right)$ is a supermartingale deflator for $\mathcal X^\varepsilon(1)$, and \cite[Theorem 4.12]{KarKar07} applies. }

\tbl{
To show \eqref{2163}, first let us fix $x>0$ and consider $X^{\pi, 0}\in\mathcal X^0(x)$, such that $$\mathbb E \left[ U(X^{\pi, 0}_T)\right] = u^0(x, 0,T)\in\mathbb R.$$ The existence of such an $X^{\pi, 0}$ follows from \eqref{2192}, no unbounded profit with bounded risk established above, and \cite[Theorem 3.2]{Mostovyi2015}. An application of Ito's lemma  shows that $\frac{X^{\pi, 0}}{L^\varepsilon} = \left( \frac{X^{\pi, 0}_t}{L^\varepsilon_t}\right)_{t\in[0,T]}\in\mathcal X^{\varepsilon}(x)$ for every $\varepsilon$ in some neighborhood of $0$. Using Assumption \ref{rra}, 
 we get
\begin{displaymath}
\begin{split}
&\left| U'\left(T, \frac{X^{\pi, 0}_T}{L^\varepsilon_T}\right)\frac{X^{\pi, 0}_T}{L^\varepsilon_T}(\bar R_T + \varepsilon \langle \bar R\rangle_T)\right|  \\
&\leq 
U'\left(T, {X^{\pi, 0}_T}\right)X^{\pi, 0}_T
\max(1,(L^\varepsilon_T)^{-\gamma_1})\frac{1}{L^\varepsilon_T}\left( |\bar R_T |+ |\varepsilon |\langle \bar R\rangle_T\right),
\end{split}
\end{displaymath}
Therefore, Assumption \ref{asPert}, implies that for every $\varepsilon$ in some neighborhood of $0$, we have 
\begin{equation}\label{2195}
\mathbb E\left[\int_0^\varepsilon\left| U'\left(T, \frac{X^{\pi, 0}_T}{L^{\tilde\varepsilon}_T}\right)\frac{X^{\pi, 0}_T}{L^{\tilde\varepsilon}_T}(\bar R_T +{\tilde\varepsilon} \langle \bar R\rangle_T)\right|d{\tilde\varepsilon}\right]<\infty.
\end{equation}
Consequently, we obtain 
\begin{equation}
\label{2202}\begin{split}
&u^\varepsilon(x,0,T) \geq \mathbb E\left[ U\left(T, \frac{X^{\pi, 0}_T}{L^{\tilde\varepsilon}_T}\right)\right] = u^0(x,0,T) +\mathbb E\left[ U\left(T, \frac{X^{\pi, 0}_T}{L^{\tilde\varepsilon}_T}\right) - U\left(T, {X^{\pi, 0}_T}\right)\right] \\ 
&= u^0(x,0,T) +\mathbb E\left[\int_0^\varepsilon U'\left(T, \frac{X^{\pi, 0}_T}{L^{\tilde\varepsilon}_T}\right)\frac{X^{\pi, 0}_T}{L^{\tilde\varepsilon}_T}(\bar R_T +{\tilde\varepsilon} \langle \bar R\rangle_T)d{\tilde\varepsilon}\right] \\
&\geq u^0(x,0,T) -\mathbb E\left[\int_0^\varepsilon\left| U'\left(T, \frac{X^{\pi, 0}_T}{L^{\tilde\varepsilon}_T}\right)\frac{X^{\pi, 0}_T}{L^{\tilde\varepsilon}_T}(\bar R_T +{\tilde\varepsilon} \langle \bar R\rangle_T)\right|d{\tilde\varepsilon}\right]>-\infty,
\end{split}
\end{equation}
where, in the second inequality, we have used \eqref{2195}. 
%
}

\tbl{
 Likewise, for a fixed $y>0$ and $Z\in\mathcal Z_0$, such that $\mathbb E \left[ V(yZ_T)\right]\in\mathbb R$, whose existence follows from \eqref{2192}, no unbounded profit with bounded risk, and \cite[Theorem 3.2]{Mostovyi2015}. An application of Ito's formula implies that for every $\varepsilon$ in some neighborhood of $0$ and $X^\varepsilon\in\mathcal X^\varepsilon(1)$, $X^\varepsilon ZL^\varepsilon = (X^\varepsilon_tZ_tL^\varepsilon_t)_{t\in[0,T]}$ is supermartingale, and thus $ZL^\varepsilon= (Z_tL^\varepsilon_t)_{t\in[0,T]}$ is a supermartingale deflator for the perturbed model. 
Therefore, similarly to \eqref{2202}, we can show that $$\infty>\mathbb E \left[ V(yZ_TL^\varepsilon_T)\right]\geq \sup\limits_{x>0}\left( u^\varepsilon(x,0,T) - xz\right).$$
}
%
%
%
\end{proof}

\begin{Theorem}\label{thmDeriv}
Let $T>0$ be fixed and suppose that Assumptions \ref{NUPBR}, \ref{finiteness}, \ref{asPert}, and  \ref{rra}  hold as well as $\mathbb E\left[U^{-}\left( T,X^{\pi,0}_T\right)\right]<\infty.$ Then, we have 
\begin{equation}\label{411}
\mathbb P-\lim\limits_{\varepsilon\to 0} J^{\eps, T}_t = J^{0,T}_t,\quad t\in[0,T].\end{equation} 
Further, for each $t\in[0,T]$, with 
 $\hat\eta_t$  being associated to $X^{\pi, 0}_t$ via \eqref{1295} and with 
\begin{equation}\label{3284}
M^R :=R^0 - \pi\cdot\langle R^0\rangle,
\end{equation}
we have
\begin{equation}\label{3285}
\begin{split}
 \mathbb P-\lim\limits_{\eps\to 0} \frac{J^{\eps, T}_t-J^{0, T}_t}{\eps} &=
X^{\pi, 0}_t\hat \eta_t\left((\psi\pi - \lambda \theta) \cdot M^R_t + {\mathbb E}^{\mathbb R^t}_t\left[(\lambda\theta)\cdot R^0_T \right]\right).
\end{split}
\end{equation}

\end{Theorem}
\textcolor{black}{
\begin{Remark}
Theorem \ref{thmDeriv} does not assert the stability nor provide the derivatives of the optimal trading strategies that are, in general, more difficult to obtain mathematically. However, Theorem \ref{thmDeriv} does show that under perturbations of the price process of the \tbl{risky asset}, the strategies that are optimal for the base model, which corresponds to $\varepsilon = 0$, drive the nearly optimal wealth processes for perturbed models. 
\end{Remark}
}
\begin{proof}[Proof of Theorem \ref{thmDeriv}] Let us fix $t\in [0,T]$.
Via a direct application of Ito's formula, one can show that
\begin{equation}\label{2201}
\begin{split}
X^{\pi,\varepsilon}_t = x\mathcal E\left(\pi\cdot R^\eps \right)_t
=x\mathcal E\left(\pi\cdot R^0 \right)_t &\frac{\mathcal E\left(\varepsilon(\psi\pi - \lambda \theta)\cdot M^R \right)_t}{\mathcal E\left(-\varepsilon(\lambda\theta)\cdot R^0 \right)_t} \\
= X^{\pi,0}_t & \frac{\mathcal E\left(\varepsilon(\psi\pi - \lambda \theta)\cdot M^R \right)_t}{L^\eps_t}.
\end{split}\end{equation}
This implies that 
$$u^\eps(X^{\pi, \varepsilon}_t, t, T) \geq \mathbb E_t\left[U\left(T,X^{\pi, 0}_t \frac{\mathcal E\left(\varepsilon(\psi\pi - \lambda \theta)\cdot M^R \right)_t \hat\rho_t}{L^\eps_T} \right)\right],\quad \mathbb P-a.s.,$$
where $\hat \rho_t$ is the optimizer to \eqref{primalProblem} corresponding to $\xi = X^{\pi, 0}_t$ and $\varepsilon = 0$, that is the base model for the \tbl{risky asset}. For $\varepsilon >0$, let us consider 
\begin{displaymath}
\begin{split}
&\frac{1}{\varepsilon}\left( u^\eps(X^{\pi, \varepsilon}_t, t, T) - u^0(X^{\pi, 0}_t, t, T)\right) \\ \geq &\frac{1}{\varepsilon}\left(\mathbb E_t\left[U\left(T,X^{\pi, 0}_t \frac{\mathcal E\left(\varepsilon(\psi\pi - \lambda \theta)\cdot M^R \right)_t \hat \rho_t}{L^\eps_T} \right) \right] -  \mathbb E_t\left[U\left(T,X^{\pi, 0}_t \hat \rho_t\right)\right]\right) \\ 
 = & \frac{1}{\varepsilon}\left(\mathbb E_t\left[\int_0^\varepsilon U'\left(T,X^{\pi, 0}_t \hat \rho_t\frac{\mathcal E\left(\beta(\psi\pi - \lambda \theta)\cdot M^R \right)_t }{L^\beta_T} \right) X^{\pi, 0}_t \hat \rho_t\frac{\partial}{\partial \beta}\left(\frac{\mathcal E\left(\beta(\psi\pi - \lambda \theta)\cdot M^R \right)_t }{L^\beta_T}\right)d\beta\right] \right).
\end{split}
\end{displaymath}
From Assumption \ref{rra}, we get
\begin{equation}\label{3281}
\begin{split}
&U'\left(T,X^{\pi, 0}_t \hat \rho_t\frac{\mathcal E\left(\beta(\psi\pi - \lambda \theta)\cdot M^R \right)_t }{L^\beta_T} \right) \\
&\leq U'\left(T,X^{\pi, 0}_t \hat \rho_t \right)\max\left(\left(\frac{\mathcal E\left(\beta(\psi\pi - \lambda \theta)\cdot M^R \right)_t }{L^\beta_T}\right)^{-\gamma_1}, 1\right).
\end{split}
\end{equation}
We recall that, in general, see e.g., \cite[Definition 1, p. 211]{ShirProb}, the definition of conditional expectation does not require integrability. This, in particular, allows to circumvent any integrability conditions on $\mathcal E\left(\beta(\psi\pi - \lambda \theta)\cdot M^R \right)_t$. Therefore, from \eqref{3281}, following \cite[Lemma 5.14]{MostovyiSirbuModel}, and using Assumption \ref{asPert}, we obtain 
\begin{equation}\label{3271}
\begin{split}
&\lim\limits_{\varepsilon\searrow 0 }\frac{1}{\varepsilon}\left( u^\eps(X^{\pi, \varepsilon}_t, t, T) - u^0(X^{\pi, 0}_t, t, T)\right) \\
&\geq \lim\limits_{\varepsilon\searrow 0 }\frac{1}{\varepsilon}\left(\mathbb E_t\left[U\left(T,X^{\pi, 0}_t \frac{\mathcal E\left(\varepsilon(\psi\pi - \lambda \theta)\cdot M^R \right)_t \hat \rho_t}{\mathcal E\left(-\varepsilon(\lambda\theta)\cdot R^0 \right)_T} \right) \right] -  \mathbb E_t\left[U\left(T,X^{\pi, 0}_t \hat \rho_t\right)\right]\right) \\ 
&=\mathbb E_t\left[U'\left(T,X^{\pi, 0}_t \hat \rho_t \right) X^{\pi, 0}_t \hat \rho_t\left((\psi\pi - \lambda \theta)\cdot M^R_t +(\lambda\theta)\cdot R^0_T\right) \right]\\
&=\mathbb E_t\left[X^{\pi, 0}_t \hat \rho_t \hat \eta_t \hat z_{tT} \left((\psi\pi - \lambda \theta)\cdot M^R_t +(\lambda\theta)\cdot R^0_T\right) \right]\\
&=X^{\pi, 0}_t\hat \eta_t\left((\psi\pi - \lambda \theta) \cdot M^R_t + {\mathbb E}^{\mathbb R^t}_t\left[(\lambda\theta)\cdot R^0_T \right]\right),\\
\end{split}
\end{equation} 
where $ \hat \eta_t$ and $\hat z_{tT}$ are given via Lemma \ref{lem1291}. 
To obtain the opposite inequality, from Lemma \ref{lem1291}, we get
$$u^\varepsilon(X^{\pi, \varepsilon}_t, t,T)\leq \mathbb E_t\left[
 V(T,\hat \eta_t\hat z_{tT} L^\eps_T)\right] + X^{\pi, \varepsilon}_t  \hat \eta_t L^\eps_t.$$
 Therefore, for $\eps>0$, using Lemma \ref{lem1291} again, we deduce that
 \begin{equation}\label{3282}
 \begin{split}
 &\frac{1}{\eps}\left(u^\varepsilon(X^{\pi, \varepsilon}_t, t,T) - u^0(X^{\pi, 0}_t, t,T)\right) \\
 &\leq
  \frac{1}{\eps}\left(\mathbb E_t\left[
 V\left(T,\hat \eta_tL^\eps_t\hat z_{tT} \frac{L^\eps_T}{L^\eps_t}\right)\right] - \mathbb E_t\left[
 V(T,\hat \eta_t\hat z_{tT})\right]\right)
 + \frac{1}{\eps}\left(X^{\pi, \varepsilon}_t  \hat \eta_t  L^\eps_t-X^{\pi, 0}_t  \hat \eta_t \right).
 \end{split}
 \end{equation}
Using Assumptions \ref{asPert} and \ref{rra}, and by passing into an $\mathcal F_t$-measurable partition of $\Omega$ as is Lemma \ref{lem4191}, we get
 \begin{equation}\label{3283}
 \lim\limits_{\varepsilon\searrow 0}\frac{1}{\varepsilon}\mathbb E_t\left[V(T,\hat \eta_t\hat z_{tT} L^\eps_T) - V(T,\hat \eta_t\hat z_{tT}) \right] = X^{\pi, 0}_t\hat \eta_t{\mathbb E}^{\mathbb R^t}_t\left[(\lambda\theta)\cdot R^0_T \right],
 \end{equation}
 and 
 \begin{equation}\label{3272}
 \begin{split}
&\lim\limits_{\varepsilon\searrow 0}\frac{1}{\eps}\left(X^{\pi, \varepsilon}_t  \hat \eta_t L^\eps_t-X^{\pi, 0}_t  \hat \eta_t  \right)\\
&= \lim\limits_{\varepsilon\searrow 0}\frac{1}{\eps}\left(X^{\pi,0}_t \frac{\mathcal E\left(\varepsilon(\psi\pi - \lambda \theta)\cdot M^R \right)_t}{L^\eps_t} \hat \eta_t L^\eps_t-X^{\pi, 0}_t  \hat \eta_t \right)\\
&= X^{\pi, 0}_t  \hat \eta_t \lim\limits_{\varepsilon\searrow 0}\frac{1}{\eps}\left( {\mathcal E\left(\varepsilon(\psi\pi - \lambda \theta)\cdot M^R \right)_t} -1\right)\\
 &
 =X^{\pi, 0}_t\hat\eta_t(\psi\pi - \lambda \theta)\cdot M^R_t.
  \end{split}
 \end{equation}
From \eqref{3271} and \eqref{3282} using \eqref{3283} and \eqref{3272}, we deduce that 
\begin{equation}\nonumber\mathbb P-\lim\limits_{\eps\searrow 0} \frac{u^\eps(X^{\pi, \eps}_t, t, T)-u^0(X^{\pi, 0}_t, t, T)}{\eps} = 
X^{\pi, 0}_t\hat \eta\left((\psi\pi - \lambda \theta) \cdot M^R_t + {\mathbb E}^{\mathbb R^t}_t\left[(\lambda\theta)\cdot R^0_T \right]\right).
\end{equation}
 Similarly, we can show that 
 \begin{equation}\nonumber
 \mathbb P-\lim\limits_{\eps\nearrow 0} \frac{u^\eps(X^{\pi, \eps}_t, t, T)-u^0(X^{\pi, 0}_t, t, T)}{\eps} = 
X^{\pi, 0}_t\hat \eta\left((\psi\pi - \lambda \theta) \cdot M^R_t + {\mathbb E}^{\mathbb R^t}_t\left[(\lambda\theta)\cdot R^0_T \right]\right).
\end{equation}
\eqref{411} and \eqref{3285} follow.
\end{proof}
\begin{Acknowledgments}
The author would like to thank Thaleia Zariphopoulou for discussions and for her suggestions on a draft of this paper. The author has been supported by the National Science Foundation under grants No. DMS-1600307 (2015 - 2019) and DMS-1848339 (2019 - 2024). Any opinions, findings, and conclusions or recommendations expressed in this material are those of the authors and do not necessarily reflect the views of the National Science Foundation.
\end{Acknowledgments}
\bibliography{literature1}\bibliographystyle{alpha}
\end{document}